\documentclass[a4paper,12pt]{amsart}
\usepackage{amsmath,inputenc,euscript,amssymb}
\usepackage{fullpage}
\usepackage{color}
\usepackage[colorlinks=true, linkcolor=magenta, citecolor=cyan, urlcolor=blue]{hyperref}
\usepackage{amsfonts} [12pt]
\font\bb=msbm9 at 12pt
\font\bbi=msbm8 at 10pt
\def\B{\hbox{\bb B}}

\def\N{\hbox{\bb N}}

\def\R{\hbox{\bb R}}

\def\Ri{\hbox{\bbi R}}

\newtheorem{theo}{Theorem}[section]
\newtheorem{pro}[theo]{Proposition}
\newtheorem{lem}[theo]{Lemma}

\newtheorem{rem}[theo]{Remark}

\title{The Keller-Segel System on the 2D-Hyperbolic Space}
\author{Patrick Maheux \&  Vittoria Pierfelice}

\address{Institut Denis Poisson (IDP)\\
Universit\'e d'Orl\'eans, Universit\'e  de Tours, CNRS Ð UMR 7013
\\ B\^atiment de Math\'ematiques, B.P. 6759\\
45067 Orl\'eans Cedex 2, France}

\email{Patrick.Maheux@univ-orleans.fr
\& Vittoria.Pierfelice@univ-orleans.fr}
\date{\today}

\begin{document}
\maketitle

\begin{abstract}
In this paper, we shall study the parabolic-elliptic Keller-Segel system on the Poincar\'e disk model of the 2D-hyperbolic space.
We shall investigate how the negative curvature of this Riemannian manifold influences the solutions of this system.
As in the 2D-Euclidean case, under the sub-critical condition $\chi M< 8\pi$, we shall prove global well-posedness results with any initial $L^1$-data.
 More precisely, by using dispersive and smoothing estimates we shall prove Fujita-Kato type theorems for local well-posedness.
 We shall then use the logarithmic Hardy-Littlewood-Sobolev estimates on the hyperbolic space to prove that the solution cannot blow-up in finite time.
 For larger mass $\chi M> 8\pi$, we shall obtain a blow-up result under an additional  condition  with respect to the flat case, probably due to the spectral gap of the Laplace-Beltrami operator.
According to the exponential growth of the hyperbolic space, we find a suitable weighted moment of exponential type on the initial data  for blow-up.
\end{abstract}

\section{Introduction}
In the last forty years, various models of the Keller-Segel system (also called Patlak-Keller-Segel) for chemotaxis have been widely studied due to important applications in biology. Most of the results of these analytical investigations focus  on the fact that the global existence or the blow-up of the solutions of these problems is a space dependent phenomenon. Historically, the key papers for this family of models are the original contribution  \cite{KellerSegel} E. F. Keller and L. A. Segel, and a work by C. S. Patlack  \cite{Patlak}. 
The optimal results in the Euclidean space $\mathbb{R}^2$ are obtained by A. Blanchet, J. Dolbeault and B. Perthame (see \cite{DP},\cite{BDP}). Moreover, a large series of results, mostly in the bounded domain case has been obtained by T. Nagai, T. Senba and T. Suzuki (see
 \cite{Nagai},\cite{Senba-Suzuki},\cite{Senba-Suzuki2},\cite{Suzuki},\cite{Jager}). The literature on this subject is huge and we shall not attempt to give a complete bibliography.

In this paper,  we study  the (parabolic-elliptic)-Keller-Segel system \eqref{kssyst} on the classical model of Riemannian manifold of constant negative curvature $-1$, namely the 2D-hyperbolic space. We present our study in the Poincar\'e disk model $\mathbb{B}^2$.
Let  
$$
\left\{
  \begin{array}{l} 
  n: [0,+\infty)  \times \mathbb{B}^2  \rightarrow \mathbb{R}^+
\\
\qquad (t,x)\rightarrow n_t(x)=n(t,x)
 \end{array}
\right.
$$
be a non-negative function satisfying the following  Keller-Segel system of equations
\begin{equation}\label{kssyst}
\left\{
  \begin{array}{l}
\frac{\partial}{\partial t} n_t(x)=\Delta_{\mathbb{H}} n_t(x)-\chi {div}_{\mathbb{H}} \left(n_t(x)\nabla_{\mathbb{H}} c_t(x)\right), \quad x\in \mathbb{B}^2, \quad t>0
\,, \vspace{.3cm}
\\
-\Delta_{\mathbb{H}} c_t(x)= n_t(x), \quad x\in \mathbb{B}^2, \quad t>0,
\\ 
n( t=0, x)=n_0(x)\geq 0,  \quad x\in \mathbb{B}^2.
  \end{array}
\right.
\end{equation}
We shall denote $n_t$ or  $n(t)$ indifferently. 
The subscript $\mathbb{H}$ on differential operators refers to the operators associated with the Riemannian metric
of the Poincar\'e disk. More details will be given below in Section \ref{formule}.
\\
We shall  understand the second condition in \eqref{kssyst} on the function $c_t$  as
$$
c_t(x)=(-\Delta_{\mathbb{H}})^{-1}n_t(x)=\int_{\mathbb{B}^2} G_{\mathbb{H}}(x,y) \, n_t(y)\, dV_y, \quad x\in \mathbb{B}^2,
$$
where $G_{\mathbb{H}}$ is  the (Dirichlet)-Green function of $-\Delta_{\mathbb{H}}$ given by
$$
G_{\mathbb{H}}(x,y)=-\frac{1}{2\pi}\log (\tanh \rho(x,y)/2),  \quad x,y \in \mathbb{B}^2.
$$
We denote by $\rho(x,y)$ the hyperbolic distance between $x$ and $y$ in $\mathbb{B}^2$ (see Section \ref{formule}). 
\\

The Cauchy problem for the analogous Keller-Segel system \eqref{kssyst} in $\mathbb{R}^2$ is now very well-understood (see \cite{DP},\cite{BDP},\cite{{BCM}}).
The natural framework in dimension two  is to work in $L^1$ which is the Lebesgue space invariant
by the scaling of the equation and   with non-negative solutions.
The mass $M=\int_{\mathbb{R}^2} n_tdx$ is then a preserved quantity. 
Nevertheless because of the scaling critical aspect of $L^1$, the conservation of the mass is not enough to ensure global well-posedness.
A simple "virial" type argument that we shall recall in Section \ref{blowupsect} allows to prove that the solution blows up
if the initial mass is such that $\chi M >8 \pi$. When $\chi M <8 \pi$, global existence results were proven in  \cite{BDP}
by using the gradient flow structure of the equation. More precisely, the Keller-Segel system can be seen
as the gradient flow of the free energy
$$ 
F[n]= \int n \, \log n -\frac{\chi}{2}\int n \,c     
$$
and thus, if the initial data has finite entropy then we get a control of the free energy for all times.
The use of the sharp logarithmic Hardy-Littlewood-Sobolev inequality in $\mathbb{R}^2$ 
then allows to get an a priori estimate in $L log L$ of the solution of \eqref{kssyst} which is  enough to propagate 
any  higher $L^p$ regularity.  In the case  $\chi M = 8 \pi$, it was latter shown in \cite{{BCM}} 
that concentration occurs in infinite time. 
        
 As we see in this brief reminder,  the results  in $\mathbb{R}^2$ are  sharp but use very deeply the structure
of the system and the dilation structure of the Euclidean space. If the system is perturbed a little bit, for example by replacing
 the Poisson equation by the equation $ -\Delta c + \alpha c = n$ with  $\alpha>0$,   or by replacing the Poisson equation by the parabolic
equation $\partial_t c- \Delta c  = n$, then the results are much less complete, see for instance \cite{CC}. 
In the same way, we can expect that any change in the geometry will also change some results.
This type of problems was already studied in bounded domains of $\mathbb{R}^2$ with various boundary conditions.
The aim of our work is to investigate the influence of the geometry and in particular the curvature
on these results.
Note that, for larger mass, an additional condition for blow-up appears with respect to the Euclidean case.
The main  blow-up result of our Theorem \ref{blowupD} will be the following. 

\begin{theo}\label{theo-A} 
There exists a weight $p: \mathbb{B}^2 \mapsto [0,+\infty[$ such that, 
if the  following  two conditions:
\begin{equation*}
\chi M>{8\pi},
\end{equation*}
with $M:=\int_{\mathbb{B}^2} n_0\,dV$, and 
\begin{equation*}
\int_{\mathbb{B}^2}p \, n_0\,dV
< M\left(\sqrt{\frac{\chi M}{8\pi}}-1\right)
\end{equation*}
are satisfied, then a smooth solution $n: [0,T^*) \times \mathbb{B}^2 \rightarrow \mathbb{R}^+$ of the Keller-Segel system  \eqref{kssyst}
with initial condition $n_0$
can exist only on a finite  interval $[0,T^*)$.
\end{theo}   
 
Note that in $ \mathbb{R}^2$, the sufficient criterion for blow-up is obtained by studying the 2-moment 
$$
I(t)= \int_{\Ri^2} |x|^2 n(t,x)\,dx.
$$
The first difficulty we had to face was to find the appropriate substitute 
for the weight $|x|^2$ on the hyperbolic space.
According to the exponential growth of the hyperbolic space, here  the good quantity to use is the following  weight of exponential type
$$
p(\rho):=p(x):=\frac{2\vert x\vert^2}{1- \vert x\vert^2}=
2\sinh^2(\rho/2)=\cosh \rho -1\geq 0, \quad x\in \mathbb{B}^2,
$$
where $\rho:=\rho(x,0)$ is the distance from $x\in\B^2$ to $0$.   
Because of the hyperbolic geometry, as we shall see in Section \ref{blowupsect}, the proof is more involved than in the Euclidean case.
Our second main result is the following global well-posedness in the case $\chi M<8 \pi$.
\begin{theo}\label{theoglob2}
For every $ n_0 \in L_+^1(\mathbb{B}^2)$,  with  $I_0= \int_{\mathbb{B}^2} p \, n_0\,dV <\infty$ and  $\chi M<8\pi$, 
we have global well-posedness on $X_{T,q} \cap \mathcal{C} (\mathbb{R}^{+},L_+^1(\mathbb{B}^2, (1+p) dV))$, 
for every $T>0$ of the Keller-Segel system  \eqref{kssyst}, where
 \begin{equation*}
X_{T,q}=
\{n : [0,+\infty) \times \mathbb{B}^2   \mapsto \R \; | \; \sup_{[0,T]} t^{(1-\frac{1}{q})}  \| n_t \|_{L^q( \mathbb{B}^2)} < + \infty \},
\end{equation*} with $\frac{4}{3} <q<2$.

\end{theo}
 A crucial ingredient in our proof is a logarithmic Hardy-Littlewood-Sobolev type inequality on the hyperbolic space, that we deduce
 from a Hardy-Littlewood-Sobolev type inequality on $\mathbb{B}^2$, see \cite{LuYa}. 
 To build the solution, by using dispersive and smoothing estimates, we shall propose a different approach than the one used in \cite{BDP} 
 which is based on the construction of weak solutions by compactness arguments. 
More precisely, we shall use the fixed point method popularized by Kato for parabolic equations (including the Navier-Stokes system)
 to prove local well-posedness  in space $X_{T,q} \cap \mathcal{C}_TL^1$, see \eqref{spacext}. We shall then use  a priori estimates
 which can be deduced from the free energy dissipation to prove that the solution cannot blow-up in finite time.
   
Note that in the hyperbolic space, the case
$$ 
\chi M >8 \pi \quad {\rm and}\quad \int_{\mathbb{B}^2}p \, n_0\,dV
\geq  M\left(\sqrt{\frac{\chi M}{8\pi}}-1\right)
$$  
is not covered by Theorem \ref{theo-A}   and Theorem \ref{theoglob2} above.
It will be interesting to analyze this case further.
A similar situation occurs in $\mathbb{R}^2$ in the case studied by V. Calvez and L. Corrias in \cite{CC}, 
where the Poisson equation is replaced by $ -\Delta c + \alpha c = n$ with  $\alpha>0$. Note that this case shares some similarities
with  our case  since the Laplacian has a spectral gap on the hyperbolic space, $\sigma(\Delta_{\mathbb{H}})=[1/4,+\infty)$. 

\section{The hyperbolic space and some useful formulas}\label{formule}

In this section, we shall recall the main geometric and analytic objects on the Poincar\'e disk and some useful formulae that we need. 
It is well-known that the Poincar\'e disk  is one of the models of the hyperbolic space, which is a non-compact Riemannian manifold with  constant negative curvature {\bf -1}.
Of course, our results can be translated in the other models as far as they are isometric to ${\mathbb{B}^2}$.
 For more details,  we refer to  Analysis and Riemannian geometry  textbooks \cite{Sto},\cite{GHL},\cite{Jost} for example.\\

Let ${\mathbb{B}^2}=\{ x\in \mathbb{R}^2, \vert x\vert<1\}$ the 2-dimensional hyperbolic disk endowed
 with its Riemannian metric
 $$
 ds^2=\frac{4(dx_1^2+dx_2^2)}{(1- \vert x\vert^2)^2}.
 $$
The hyperbolic distance between $x\in {\mathbb{B}^2}$ and $0=(0,0)\in {\mathbb{B}^2}$ is given by
$$
\rho:=\rho(x,0)=\log \left(\frac{1+\vert x\vert}{1-\vert x\vert}\right)
$$
(equivalent to
$
\vert x\vert=\tanh (\rho/2)
$). More generally, the hyperbolic distance between $x$ and $y$ in the disk  is given by
$$
\rho(x,y)=
\rho(T_x(y),0)=\rho(T_y(x),0)=\log \left(\frac{1+\vert T_x(y)\vert}{1-\vert  T_x(y)\vert}\right),
$$
where $T_x(y)$ is the M\"oebius transformation 
$$
T_x(y)=\frac{\vert y-x\vert^2x-(1-\vert x\vert^2)(y-x)}
{1-2x \cdot y+\vert x\vert^2\vert y\vert^2},
$$
with $x \cdot y=x_1y_1+x_2y_2$ denoting the scalar product on $\mathbb{R}^2$ (see \cite{Sto}). We have several useful relations
$$
\vert T_x(y)\vert^2=
 \frac{\vert x-y\vert^2}
{1-2x \cdot y+\vert x\vert^2\vert y\vert^2},
$$
$$
T_y(T_y(x))=x,
$$
$$
\sinh \frac{\rho(T_x(y))}{2}
=
\frac{\vert T_x(y)\vert }{\sqrt{1-\vert T_x(y)\vert^2}}
=
\frac{\vert x-y\vert}{\sqrt{(1-\vert x\vert^2)(1-\vert y\vert^2)}},
$$
$$
\cosh \frac{\rho(T_x(y))}{2}
=
\frac{1}{\sqrt{1-\vert T_x(y)\vert^2}}
=
\frac{\sqrt{1-2x  \cdot y+\vert x\vert^2\vert y\vert^2}}{\sqrt{(1-\vert x\vert^2)(1-\vert y\vert^2)}}.
$$
 We can consider three different systems of coordinates:
$x=(x_1,x_2)$ (cartesian coordinates), $(r,\theta)$ with $r=\vert x \vert$ (polar coordinates) and $(\rho,\theta)$
(spherical coordinates) with
$
\vert x\vert=\tanh (\rho/2)
$
where $\vert x\vert $ is the  Euclidean norm.
So, we can write 
$$
x=\vert x\vert(\cos \theta,\sin \theta)
=\vert x\vert e^{i\theta}
=\tanh (\rho/2)e^{i\theta}=
\tanh (\rho/2)(\cos \theta,\sin \theta) .
$$
We shall denote indifferently
$f(x)=f(r,\theta)=f(re^{i\theta})=f(\rho, \theta)$ for simplicity. 
We shall also  use indifferently  in the same equation both variables $x$ and $\rho$.

Let us  denote by $g_x(X,Y)$ 
the  metric tensor on  two vector fields  
$$X(x)=X_1(x)\frac{\partial }{\partial x_1} +X_2(x)\frac{\partial }{\partial x_2}, \quad
Y(x)=Y_1(x)\frac{\partial }{\partial x_1} +Y_2(x)\frac{\partial }{\partial x_2},$$
evaluated at $x\in {\mathbb{B}^2}$, which is given  by
$$
g_x(X,Y)=\left(\frac{2}{1-\vert x\vert^2}\right)^2 \sum_{i=1}^2 X_iY_i.
$$
The Riemannian element of volume (measure) of the exponential growth is given by
$$
dV(x)=
\left( \frac{2}{1- \vert x\vert^2}\right)^2dx
=\sinh \rho \, d\rho d\theta,
$$
where $dx=dx_1dx_2$ is the Lebesgue measure on ${\mathbb{B}^2}$.
In cartesian coordinates, we now define classical differential operators. 
First of all, the gradient $\nabla_{\mathbb{H}}$ 
with respect to the Riemannian structure is defined by
$$
\nabla_{\mathbb{H}} f(x):=\left(\frac{1- \vert x\vert^2}{2}\right)^2 \, \nabla_e f(x),\quad x\in {\mathbb{B}^2},
$$
where $\nabla_e f=(\frac{\partial f}{\partial x_1}, \frac{\partial f}{\partial x_2})$
is the Euclidean gradient. For a radial function $f$, we  note $f(x)=f(\rho(x))$ and have
$$
\nabla_{\mathbb{H}} \left( f(\rho(x)) \right)
=
\frac{f^{\prime}(\rho)}{2\cosh^2 (\rho/2)}\frac{x}{\vert x\vert}
=
\frac{f^{\prime}(\rho)}{\sinh (\rho)}x,
$$
where $f^{\prime}(\rho)$ is the derivative of $f$ with respect to $\rho$.
In particular
$$
g_x(\nabla_{\mathbb{H}} f,\nabla_{\mathbb{H}} f)=\left(\frac{1-\vert x\vert^2}{2}\right)^2
\nabla_e f \cdot \nabla_e f,
$$
where $\nabla_e f \cdot \nabla_e g=\frac{\partial f}{\partial x_1}\frac{\partial g}{\partial x_1}+
\frac{\partial f}{\partial x_2}\frac{\partial g}{\partial x_2}$.
For a radial function $f$, we have
$$
\vert \nabla_{\mathbb{H}} f\left(\rho(x)\right) \vert_g^2
=
g_x(\nabla_{\mathbb{H}} f,\nabla_{\mathbb{H}} f)
=
\vert f^{\prime}(\rho)\vert^2.
$$
The divergence of the vector field
$
Z(x)=Z_1(x)\frac{\partial }{\partial x_1} +Z_2(x)\frac{\partial }{\partial x_2} 
$
is defined  by
$$
{div}_{\mathbb{H}} Z (x) =\nabla_{\mathbb{H}} \cdot Z (x) :=\frac{4}{(1-\vert x\vert^2)} \sum_{i=1}^2 x_iZ_i(x)
+ \sum_{i=1}^2 \partial_i Z_i(x), \quad x\in {\mathbb{B}^2}.
$$
We also define the Laplace-Beltrami operator on ${\mathbb{B}^2}$
$$
\Delta_{\mathbb{H}}f(x):=
\left( \frac{1- \vert x\vert^2}{2}\right)^2 
\Delta_e f(x),\quad x\in  {\mathbb{B}^2},
$$
where $\Delta_e $ is the Euclidean Laplacian
$$
\Delta_e=
\frac{\partial^2}{\partial^2 x_1}+\frac{\partial^2}{\partial^2 x_2},
$$
with  $x=(x_1,x_2)\in  {\mathbb{B}^2}$.
For radial function $f(\rho)$, the  Laplace-Beltrami operator  takes the form
$$\Delta_{\mathbb{H}}f(\rho)
=f^{\prime\prime}(\rho)+\coth(\rho)f^{\,\prime}(\rho)
=\frac{1}{\sinh \rho}\frac{\partial}{\partial \rho} 
\left(\sinh \rho \frac{\partial}{\partial \rho} f\right).
$$ 
The (Dirichlet)-Green function  of $-\Delta_{\mathbb{H}}$ is given by
\begin{equation}\label{green}
G_{\mathbb{H}}(x,y)=-k_1\log \tanh (\rho(x,y)/2)
=
-\frac{k_1}{2} \log \vert T_x(y)\vert^2,
\end{equation}
with $k_1=\frac{1}{2\pi}$
(see \cite{Sto}).
\\

For any function $u$ and vector field $Y$ defined  on ${\mathbb{B}^2}$, we define in the usual way
$$ 
\| u\|_{L^q}= \left(\int_{\mathbb{B}^2}  |u(x)|^q  dV_x \right)^{ 1 \over q},\, \quad
 \| Y\|_{L^q}= \left(\int_{\mathbb{B}^2}  |g_x(Y, Y)|^{q\over 2}  dV_x \right)^{ 1 \over q}.
$$

\section{Local well-posedness}
\subsection{Auxiliary estimates}
Let us first state the dispersive and smoothing estimates that we shall use for the heat equation associated with the Laplace-Beltrami operator. The proof of these estimates can be found for example by V. Pierfelice in \cite{Pierfe}.
\begin{lem}
\label{disp}
For every $ 1 \leq p \leq q$, we have the estimate
$$
   \|e^{t \Delta_{\mathbb{H}}} u_0 \|_{L^q(\mathbb{B}^2)} \leq  c_1(t)^{ { 1 \over p } - { 1 \over q} }  e^{- t \gamma_{p,q} }  \|u_0\|_{L^p(\mathbb{B}^2)}, \quad t>0,
   $$
 where $\gamma_{p,q}= \delta ( { 1 \over p } - { 1 \over q}  + {8 \over q} ( 1- { 1 \over p}))\geq0$  and
 $c_1(t)= C \mbox{Max }( 1, t^{-1})$ for some $\delta$, $C >0$.
\end{lem}
We shall also use the next lemma.
\begin{lem}
\label{smooth}
For every $ 1 \leq p \leq q$,  and every vector field $Y$ on $\mathbb{B}^2$, we have the estimate
$$
   \|e^{t \Delta_{\mathbb{H}}} \mbox{div}_{\mathbb{H }}Y \|_{L^q(\mathbb{B}^2)} \leq  c_1(t)^{ { 1 \over p } - { 1 \over q}  + { 1 \over 2 }}  e^{- t {\gamma_{p,q}  + \gamma_{q,q}   \over 2 } }  \|Y\|_{L^p(\mathbb{B}^2)}, \quad t>0,
   $$
 where $\gamma_{p,q}= \delta ( { 1 \over p } - { 1 \over q}  + {8 \over q} ( 1- { 1 \over p}))\geq0$  and
 $c_1(t)= C \mbox{Max }( 1, t^{-1})$ for some $\delta$, $C >0$.
\end{lem}

We shall also need to estimate $c:= -\Delta_{\mathbb{H}}^{-1} n.$ To do so, in the next Lemma, we use the boundedness of the Riesz transorm, Sobolev embedding and Poincar\'e inequality on the hyperbolic space.
\begin{lem}
\label{elliptic}
  For every $s >1$, $1<q<2$ such that ${ 1 \over s } = { 1 \over q} - { 1 \over 2},$ we have the estimate
  $$ \|\nabla_\mathbb{H}  c \|_{L^s(\mathbb{B}^2)} \lesssim  \|n\|_{L^q(\mathbb{B}^2)}.$$
\end{lem}

\begin{proof}
 We want to estimate $ \nabla_\mathbb{H}  c =   -\nabla_\mathbb{H}   \Delta_{\mathbb{H}}^{-1} n$.
  Note that we can write $$ \nabla_\mathbb{H}  c =  -(\nabla_\mathbb{H}    \Delta_{\mathbb{H}}^{-{1 \over 2} } )  \Delta_{\mathbb{H}}^{-{1 \over 2} }n$$ therefore, by using the continuity of the Riesz transform on $L^s$ (see \cite{Strichartz} for example), we obtain
  $$ \|  \nabla_\mathbb{H}  c  \|_{L^s} \lesssim  \| \Delta_{\mathbb{H}}^{-{1 \over 2} }n\|_{L^s}.$$
  Next, by using the Sobolev embedding $W^{1, q}\subset L^s$ and the Poincar\'e inequality on the hyperbolic space,  we obtain
  $$   \|  \nabla_\mathbb{H}  c  \|_{L^s} \lesssim   \| \Delta_{\mathbb{H}}^{-{1 \over 2} }n\|_{W^{1, q}}
   \lesssim    \| \nabla_H \Delta_{\mathbb{H}}^{-{1 \over 2} }n\|_{L^q}.$$
    By using again the continuity of the Riesz transform on $L^q$, we finally obtain that
    $$   \|  \nabla_\mathbb{H}  c  \|_{L^s} \lesssim  \|n\|_{L^q},$$
     which is the desired estimate.
    
\end{proof}

\subsection{Local well-posedness results}

Let $1\leq q\leq +\infty$ and $T>0$ be fixed. We recall the following Banach space
\begin{equation}\label{spacext}
X_{T,q}=
\{n:  [0,+\infty) \times \mathbb{B}^2  \mapsto \R \; | \; \sup_{[0,T]} t^{(1-\frac{1}{q})}  \| n_t \|_{L^q( \mathbb{B}^2)} < + \infty \},
\end{equation}
with norm
$$ 
\| n\|_{X_{T,q}} = \sup_{[0,T]} t^{(1-\frac{1}{q})}  \| n_t \|_{L^q( \mathbb{B}^2)}.
$$

\begin{theo}
\label{theol1}
 For every $ n_0 \in L^1$  and $\frac{4}{3}<q<2$, there exists $T>0$ such that there  is  a unique solution $n$ 
 of the Keller-Segel system  \eqref{kssyst} in $X_{T,q} \cap \mathcal{C} ([0, T],L^1(\mathbb{B}^2))$.
\end{theo}

\begin{rem}
Note that as usual,  the arguments in the proof can also be used to get global well-posedness for sufficiently small data in $L^1$.
 \end{rem}

\begin{proof}
 Without lost of generality, we can assume  that $\chi=1$ in \eqref{kssyst}.
 By using Duhamel formula,  solving the system \eqref{kssyst} is equivalent to look for
 \begin{equation}\label{solutionDuhamel}
  n = e^{ t  \Delta_\mathbb{H} } n_0 + B(n,n),
  \end{equation}
  with
$$ 
B(n,n) = - \int_0^t  e^{( t- s)  \Delta_\mathbb{H} } \mbox{div}_{\mathbb{H}} ( n  (s) \nabla_{\mathbb{H}} c (s) )\, ds, \quad \quad
c =  - \Delta_{\mathbb{H}}^{-1} n.
$$
We shall use the following classical variant of the  Banach fixed point Theorem:

\begin{lem}\label{lemmafixed}
 Consider $X$  a Banach space and $B$ a bilinear operator $X\times X \longrightarrow X$ such that
 {$$ \forall u, \, v \in X, \quad \|B(u,v)\|_{X} \leq \gamma \|u\|_{X}\, \|v\|_{X}, $$}
 then, for every $u_{1}\in X$, such that  {$ 4 \gamma \|u_{1}\|_{X}<1$,} the sequence defined by 
 $$  u_{n+1} =u_{1}+ B(u_{n}, u_{n}), \quad u_{0}=0$$
 converges to {the unique solution} of 
 {$$ u=  u_{1} + B(u,u)$$}
such that {$2 \gamma  \|u\|_{X} <1$.}
\end{lem}

We can always assume that $T \leq 1$ such that, in the following, 
we will use Lemma \ref{disp} and Lemma \ref{smooth} with $c(t)= C/t$.
\\

{\it Step 1. Existence of solutions in ${X_{T,q}}$}.
\\
{\it  1.1. Smallness of $u_1=   e^{ t  \Delta_\mathbb{H} } n_0 $ in $X_{T,q}$ for some $T>0$}.
On one hand,  by applying Lemma \ref{disp}, we have 
\begin{equation}\label{boundonneinitial}
N_T(n_0):=\| e^{ t  \Delta_\mathbb{H} } n_0 \|_{X_{T,q}} \leq c \| n_0 \|_{L^1( \mathbb{B}^2)},
\end{equation} 
for any $n_0 \in {L^1( \mathbb{B}^2)}$ and $q\geq 1$.
On the other hand,  we have 
$$
\lim_{t\rightarrow 0} t^{(1-\frac{1}{q})} \| e^{ t  \Delta_\mathbb{H} } n_0 \|_{L^q( \mathbb{B}^2)}=0,
$$
for any $n_0\in L^q\cap L^1$ and  $q>1.$
Hence  $\lim_{T\rightarrow 0} N_T(n_0)=0$.
Then, by density of $L^q\cap L^1$ in $L^1$, we get 
\begin{equation}\label{donneinitial}
\lim_{T\to 0} \| e^{ t  \Delta_\mathbb{H} } n_0 \|_{X_{T,q}} =0,
\end{equation} 
for any $n_0\in {L^1( \mathbb{B}^2)}$.
So, for all $\gamma>0$ and all $n_0\in {L^1( \mathbb{B}^2)}$, there exists $T>0$ such that $ 4\gamma \| e^{ t  \Delta_\mathbb{H} } n_0 \|_{X_{T,q}}<1$.
\\

 \noindent{\it 1.2. Boundedness of $B(n,n)$}.
Next we shall study the continuity of  $B(n,n) $ on $X_{T,q}$. By using successively Lemma \ref{smooth} and the H\"older inequality, we obtain
\begin{align*}
 \| B(n,n) \| _{X_{T,q}} & \lesssim   \sup_{[0,T]} t^{(1-\frac{1}{q})}  \int_0^t  \frac{1}{(t-s)^{\frac{1}{p}-\frac{1}{q}+\frac{1}{2}}} \| n(s) \nabla_{\mathbb{H}} \Delta^{-1}_{\mathbb{H}} n(s) \|_{L^p( \mathbb{B}^2)} \, ds \\ 
& \lesssim  \sup_{[0,T]} t^{(1-\frac{1}{q})}  \int_0^t  \frac{1}{(t-s)^{\frac{1}{p}-\frac{1}{q}+\frac{1}{2}}} \| n(s)  \|_{L^q( \mathbb{B}^2)}  \| \nabla_{\mathbb{H}} \Delta^{-1}_{\mathbb{H}} n(s) \|_{L^r( \mathbb{B}^2)} \, ds, 
\end{align*}
 with $\frac{1}{r} =\frac{1}{p}-\frac{1}{q}$ and $ p\leq q.$ By using  Lemma \ref{elliptic} with $c = - \Delta^{-1}_{\mathbb{H}} n$, we have
 \begin{align*}
& \lesssim   \sup_{[0,T]} t^{(1-\frac{1}{q})}  \int_0^t  \frac{1}{(t-s)^{\frac{1}{p}-\frac{1}{q}+\frac{1}{2}}} \| n(s)  \|^2_{L^q( \mathbb{B}^2)}
 \, ds, 
\end{align*}
 with $\frac{1}{r} =\frac{1}{q}-\frac{1}{2}.$ Thus for  $1 \leq p\leq q$ such that $\frac{1}{p} = \frac{2}{q} -\frac{1}{2}$ and hence $q\geq \frac{4}{3}$, we have 
 \begin{equation}
  \| B(n,n) \| _{X_T} \lesssim I_1   \| n \|^2_{X_T},
  \end{equation}
  where 
  \begin{equation}
  I_1= \sup_{[0,T]} t^{(1-\frac{1}{q})}  \int_0^t  \frac{1}{(t-s)^{\frac{1}{p}-\frac{1}{q}+\frac{1}{2}}} \frac{1}{s^{2(1-\frac{1}{q})}} \, ds.
  \end{equation}
 By using the change of variables $s=t w$, we find 
 \begin{equation*}
  I_1 = \sup_{[0,T]} t^{(1-\frac{1}{q})}  \int_0^1  \frac{1}{t^{\frac{1}{p}-\frac{1}{q}+\frac{1}{2}} (1-w)^{\frac{1}{p}-\frac{1}{q}+\frac{1}{2}}} \frac{t  \, dw}{t^{2(1-\frac{1}{q})} w^{2(1-\frac{1}{q})}} =   
  \int_0^1  \frac{1}{(1-w)^{\frac{1}{q}}} \frac{1}{ w^{2(1-\frac{1}{q})}}  \, dw
 \end{equation*}
 assuming $\frac{1}{q}< 1$ i.e. $q>1$ and $2(1-\frac{1}{q}) <1$ i.e. $q<2$, we ensure $I_1 <\infty$
and hence 
   \begin{equation}\label{contraz}
  \| B(n,n) \| _{X_{T,q}} \lesssim  I_1   \| n \|^2_{X_{T,q}}, \quad \mathrm{for} \,\, \frac{4}{3}\leq q<2.
  \end{equation}
 \noindent
 {\it 1.3. Conclusion}. By using Lemma \ref{lemmafixed} with $X=X_{T,q}$ with $T>0$ given by Step 1.1, \eqref{donneinitial} and \eqref{contraz}, we find a solution of \eqref{kssyst} in $X_{T,q}$ i.e for Keller-Segel problem with initial data $n_0$ in $L^1$ for $T$ sufficiently small.
\\

{\it Step 2. Proof of $n \in {L^{\infty} ([0, T],L^1( \mathbb{B}^2))}$ and Conclusion}.
\\
 Since by Lemma \ref{disp} we have 
  \begin{equation}\label{donneinitialL1}
 \| e^{ t  \Delta_\mathbb{H} } n_0 \|_{L^{\infty} ([0, T],L^1( \mathbb{B}^2))} \leq    \| n_0 \|_{L^1( \mathbb{B}^2)}, 
\end{equation} 
it remains to estimate $ \|  B(n,n)  \|_{L^{\infty} ([0, T],L^1( \mathbb{B}^2)}.$ 
By using successively Lemma \ref{smooth} and the H\"older inequality, we obtain
\begin{align*}
 \| B(n,n) \| _{L^{\infty} ([0, T],L^1( \mathbb{B}^2))} &Ê\lesssim  \sup_{[0,T]}   \int_0^t  \frac{1}{(t-s)^{\frac{1}{2}}} \| n(s) \nabla_{\mathbb{H}} \Delta^{-1}_{\mathbb{H}} n(s) \|_{L^1( \mathbb{B}^2)} \, ds \\ 
& \lesssim   \sup_{[0,T]}  \int_0^t  \frac{1}{(t-s)^{\frac{1}{2}}} \| n(s)  \|_{L^q( \mathbb{B}^2)}  \| \nabla_{\mathbb{H}} \Delta^{-1}_{\mathbb{H}} n(s) \|_{L^{q'}( \mathbb{B}^2)} \, ds. 
\end{align*}
Thus we have 
\begin{equation}
 \| B(n,n) \| _{L^{\infty} ([0, T],L^1( \mathbb{B}^2))} Ê\lesssim   \sup_{[0,T]} \left(  \int_0^t  \frac{1}{(t-s)^{\frac{1}{2}} s^{1-\frac{1}{q}}}
  \| \nabla_{\mathbb{H}} \Delta^{-1}_{\mathbb{H}} n(s) \|_{L^{q'}( \mathbb{B}^2)} \, ds \right) \| n \|_{X_{T,q}}.
 \end{equation}
 By Lemma \ref{elliptic} we have
 $$  \| \nabla_{\mathbb{H}} \Delta^{-1}_{\mathbb{H}} n(s) \|_{L^{q'}( \mathbb{B}^2)} \lesssim  \| n(s)  \|_{L^{\eta}( \mathbb{B}^2)}, \quad \mathrm{with} \,\, \frac{1}{\eta}=\frac{3}{2}-\frac{1}{q}.$$ 
 Since  $\frac{4}{3}< q < 2$, we obtain that $1 < \eta<q$ and hence we can use the following interpolation inequality
$$ \| n(s)  \|_{L^{\eta}( \mathbb{B}^2)} \lesssim   \| n(s)  \|^{\theta}_{L^{1}( \mathbb{B}^2)}   \| n(s)  \|^{1- \theta}_{L^{q}( \mathbb{B}^2)}, \quad \mathrm{with} \,\, \theta = \frac{(\frac{3}{2}-\frac{2}{q})}{(1-\frac{1}{q})}$$ 
 to obtain 
 \begin{equation}
 \| B(n,n) \| _{L^{\infty}  ([0, T],L^1( \mathbb{B}^2))} Ê\lesssim    \left(  \sup_{[0,T]} \int_0^t  \frac{1}{(t-s)^{\frac{1}{2}} s^{1-\frac{1}{q}}}
\| n(s)  \|^{\theta}_{L^{1}( \mathbb{B}^2)}   \| n(s)  \|^{1- \theta}_{L^{q}( \mathbb{B}^2)}    \, ds \right) \| n \|_{X_{T,q}}.
 \end{equation}
  Thus we have 
 \begin{equation}
  \| B(n,n) \| _{L^{\infty}  ([0, T],L^1( \mathbb{B}^2))} \lesssim I_2   \| n \|_{ L^{\infty}  ([0, T],L^1( \mathbb{B}^2))}^{\theta}  \| n \|^{2-\theta}_{X_{T,q}}, 
  \end{equation}
  where 
  \begin{equation}
  I_2= \sup_{[0,T]}  \int_0^t  \frac{1}{(t-s)^{\frac{1}{2}}} \frac{1}{s^{(2-\theta)(1-\frac{1}{q})}} \, ds =  \sup_{[0,T]}  \int_0^t  \frac{ds}{(t-s)^{\frac{1}{2}} s^{\frac{1}{2}}}. 
  \end{equation}
As before,  by using the change of variables $s=t w$, we find 
  \begin{equation*}
  I_2=   \sup_{[0,T]}  \int_0^1  \frac{t dw}{ t \,(1-w)^{\frac{1}{2}} w^{\frac{1}{2}}} 
  \end{equation*}
  which is finite.
  This yields 
  $$
    \| n \| _{L^{\infty}  ([0, T],L^1( \mathbb{B}^2))} \lesssim    \| n_0 \| _{L^1( \mathbb{B}^2)} +  
    \| n \|_{L^{\infty}  ([0, T],L^1( \mathbb{B}^2))}^{\theta}  \| n \|^{2-\theta}_{X_{T,q}},  $$
with $0<\theta<1$.
By using the Young inequality we have
  $$
    \| n \| _{L^{\infty}  ([0, T],L^1( \mathbb{B}^2))} \lesssim    \| n_0 \| _{L^1( \mathbb{B}^2)} + \| n \|^{\frac{2-\theta}{1-\theta}}_{X_{T,q}},  $$
 which proves that $n \in L^{\infty} ([0, T],L^1( \mathbb{B}^2)) \cap X_{T,q}$. By a classical  argument we deduce that $n \in \mathcal{C}( [0, T],L^1( \mathbb{B}^2)) \cap X_{T,q}$.
 
 We can also deduce the uniqueness of $n \in  \mathcal{C} ([0, T],L^1( \mathbb{B}^2)) \cap X_{T,q}$ from similar arguments. This produces automatically local well-posedeness result.
    
  \end{proof}
  
Similar to Theorem \ref{theol1}, we shall obtain the next result of existence and uniqueness  of the solution of the 
Keller-Segel system  \eqref{kssyst} with initial data $n_0$ in  $L^q(\mathbb{B}^2)$ for  all $q$ such that $\frac{4}{3}\leq q<2$ 
and  a uniform positive lower bound on the existence time $T$ independent of   
the initial data $n_0$ in any fixed   ball of  $L^q$.
 
  \begin{theo}
\label{theol2}
 For every $ n_0 \in L^q(\mathbb{B}^2)$ and $\frac{4}{3}\leq q<2$,  there exists $T>0$ such that  there exists a unique solution 
 $n \in  \mathcal{C} ([0, T],L^q(\mathbb{B}^2))$ of the Keller-Segel system  \eqref{kssyst}. 
 Moreover for every $R>0$ there exists $T(R) >0$ such that $T \geq T(R)$ for every $ \|  n_0 \|_{L^q(\mathbb{B}^2)} \leq R.$
  \end{theo}
   \begin{proof}
We shall again use  Lemma \ref{lemmafixed} to solve \eqref{kssyst}.
We can always assume that $T \leq 1$ such that in the following we will use Lemma \ref{disp} 
and Lemma \ref{smooth} with $c(t)= C/t$.
We have 
\begin{equation}\label{donneinitial2}
 \| e^{ t  \Delta_\mathbb{H} } n_0 \|_{ \mathcal{C} ([0, T],L^q(\mathbb{B}^2))} \leq   \|n_0 \|_{L^q( \mathbb{B}^2)}.
\end{equation} 
To study the continuity of  $B(n_t,n_t) $ on $\mathcal{C} ([0, T],L^q(\mathbb{B}^2))$ we use successively Lemma \ref{smooth} and the H\"older inequality
\begin{align*}
 \| B(n,n) \| _{\mathcal{C} ([0, T],L^q(\mathbb{B}^2))} &
\lesssim 
   \sup_{[0,T]}  \int_0^t  \frac{1}{(t-s)^{\frac{1}{p}-\frac{1}{q}+\frac{1}{2}}} \| n(s) \nabla_{\mathbb{H}} \Delta^{-1}_{\mathbb{H}} n(s) \|_{L^p( \mathbb{B}^2)} \, ds \\ 
& \lesssim  \sup_{[0,T]}  \int_0^t  \frac{1}{(t-s)^{\frac{1}{p}-\frac{1}{q}+\frac{1}{2}}} \| n(s)  \|_{L^q( \mathbb{B}^2)}  \| \nabla_{\mathbb{H}} \Delta^{-1}_{\mathbb{H}} n(s) \|_{L^r( \mathbb{B}^2)} \, ds, 
\end{align*}
 with $\frac{1}{r} =\frac{1}{p}-\frac{1}{q}, \, p\leq q.$ By using  Lemma \ref{elliptic} with $c =  -\Delta^{-1}_{\mathbb{H}} n$, we have
 \begin{align*}
& \lesssim   \sup_{[0,T]}   \int_0^t  \frac{1}{(t-s)^{\frac{1}{p}-\frac{1}{q}+\frac{1}{2}}} \| n(s)  \|^2_{L^q( \mathbb{B}^2)}
 \, ds, 
\end{align*}
 with $\frac{1}{r} =\frac{1}{q}-\frac{1}{2}$, so $\frac{4}{3}\leq q<2$. Thus for  $1 \leq p\leq q$ such that $\frac{1}{p} = \frac{2}{q} -\frac{1}{2}$, we have 
 \begin{equation}
  \| B(n,n) \| _{\mathcal{C} ([0, T],L^q(\mathbb{B}^2))} \lesssim I_3   \| n \|^2_{\mathcal{C} ([0, T],L^q(\mathbb{B}^2))},
  \end{equation}
  where 
  \begin{equation}
  I_3= \sup_{[0,T]} \int_0^t  \frac{1}{(t-s)^{\frac{1}{p}-\frac{1}{q}+\frac{1}{2}}}  \, ds.
  \end{equation}
As before, by using the change of variables $s=t w$, we find 
  \begin{equation*}
  I_3 = \sup_{[0,T]}  \int_0^1  \frac{t dw}{t^{\frac{1}{p}-\frac{1}{q}+\frac{1}{2}} (1-w)^{\frac{1}{p}-\frac{1}{q}+\frac{1}{2}}} =   \sup_{[0,T]}  
  t^{(1-\frac{1}{q})}  \int_0^1  \frac{1}{(1-w)^{\frac{1}{q}}} \, dw
  \end{equation*}
 which is finite for  $q>1$.
Hence there exists a constant $C_1>0$ independent of $T$ and $n$ such that
   \begin{equation}\label{contrazione}
  \| B(n,n) \| _{\mathcal{C} ([0, T],L^q(\mathbb{B}^2))} 
  \leq C_1   T^{(1-\frac{1}{q})} \| n\|^2_{\mathcal{C} ([0, T],L^q(\mathbb{B}^2))}
  \end{equation}
  for any $q$ such that $\frac{4}{3}\leq q<2$.
  
Let $X={\mathcal{C} ([0, T],L^q(\mathbb{B}^2))}$  and 
$\gamma:=  C_1   T^{(1-\frac{1}{q})}$. Imposing the conditions $0<T\leq 1$ and 
$4 C_1   T^{(1-\frac{1}{q})}\| n_0 \|_{L^q(\mathbb{B}^2)}<1$, this implies 
$4\gamma \| e^{ t  \Delta_\mathbb{H} } n_0 \|_{ \mathcal{C} ([0, T],L^q(\mathbb{B}^2))}<1$
by \eqref{donneinitial2}.
By using Lemma \ref{lemmafixed} with $u_1:=e^{ t  \Delta_\mathbb{H} } n_0$ and Duhamel formula \eqref{solutionDuhamel}, we obtain a unique solution $n \in {\mathcal{C} ([0, T],L^q(\mathbb{B}^2))}$  
of the Keller-Segel system \eqref{kssyst}. This produces automatically local well-posedeness result.
 \\
\indent
We now prove the uniformity result.
Choose $T(R):=(8C_1R)^{\frac{-q}{q-1}}$ such that $T(R)\leq 1$ for $R$ positive and  large enough. 
The same arguments as above prove the existence of a unique solution $n$ defined on 
the  fixed interval $[0,T(R)]$ for all initial conditions $n_0\in L^q(\mathbb{B}^2)$ such that
 $\| n_0 \|_{L^q(\mathbb{B}^2)}\leq R$.
This finishes the proof of our theorem.
 \end{proof}

\section{Blow-up}\label{blowupsect}
In the case of $\mathbb{R}^2$, the blow-up for the Keller-Segel system is quite easy to prove.
In fact, under the assumption  $\int_{\mathbb{R}^2} (1+ \vert x\vert^2)n_0(x)\,dx<\infty$, we have the following "virial" type identity
\begin{equation}\label{momentR2}
\int_{\mathbb{R}^2} \vert x\vert^2n(t,x)\,dx
=
\int_{\mathbb{R}^2} \vert x\vert^2n_0(x)\,dx
+\frac{4M}{8\pi}(8\pi -\chi M)t, \quad \forall t>0,
\end{equation}
where $M= \int_{\mathbb{R}^2} n_0(x)\,dx$. If $\chi M>8 \pi$ and $t$ large enough, 
the right-hand side of \eqref{momentR2} is negative, contradicting the non-negative left-hand side of the equation. 
Thus the solution cannot exist for $t>T^*$ for some finite $T^*$.
\\

In this section, our goal is to study the blow-up  phenomenon for the solution of \eqref{kssyst} on ${{\mathbb{B}^2}}$. 
Because of the geometry of the hyperbolic space, our main difficulty here is to find an appropriate weight to obtain 
a "virial"  type argument for blow-up. Thanks to our choice of a weight of exponential type, we are able to replace the identity
\eqref{momentR2} by the inequality \eqref{weightinequality} below.  This inequality will allow to prove blow up
 for $M=\int_{{\mathbb{B}^2}} n_0(x)\, dV >{8\pi}/{\chi }$ under an additional condition on the moment, with a suitable weight $p$ of exponential type, $\int_{{\mathbb{B}^2}}  p \, n_0 dV$. 
As noted before, an additional condition for blow-up  on the 2-moment  $\int_{\mathbb{R}^2} \vert x\vert^2n_0\,dx$ was also needed in 
\cite{CC}[Thm.1.2 eq.(1.2)] where in particular  the Keller-Segel system on $\mathbb{R}^2$ with the Laplacian  replaced
 by the operator $-\Delta + \alpha$, $\alpha>0$ was studied.

%
Let us  recall the expression of the  weight $p$ that we shall use in our blow-up argument
\begin{equation}\label{weightnatur}
p(\rho):=p(x):=\frac{2\vert x\vert^2}{1- \vert x\vert^2}=
2\sinh^2(\rho/2)=\cosh \rho -1\geq 0, \quad x\in {\mathbb{B}^2}.
\end{equation}
Note that the expression of the weight $p=2\sinh^2(\rho/2)=\cosh \rho -1$ 
is the same in any isometric representation of the Poincar\'e disk
(for instance in the Poincar\'e upper-half space model).
Next, we shall need the following relations
$
\frac{p}{2}+1=\frac{1}{1- \vert x\vert^2}
$
and 
\begin{equation}\label{laplace p}
\Delta_{{\mathbb{H}}}\, p=2p+2.
\end{equation}

Our blow-up result is the following statement.

\begin{theo}\label{blowupD}
Let $n: [0,T^*)  \times {\mathbb{B}^2} \rightarrow  {\mathbb{R}^+}$ be a solution of the Keller-Segel system  \eqref{kssyst}
with $T^*\leq +\infty$ such that $n \in  \mathcal{C} ([0, T^*),L_+^1(\mathbb{B}^2, (1+p) dV)$. 
Then we have 
\begin{enumerate}
\item
For all $t\in [0,T^*)$,
\begin{equation}\label{weightinequality}
\left(\int_{{\mathbb{B}^2}} p \, n_t\; dV +M\right)^2
\leq
\left(\left[\int_{{\mathbb{B}^2}} p \, n_0\;dV +M\right]^2-\frac{\chi}{8\pi}M^3\right)e^{4t}
+
\frac{\chi}{8\pi}M^3,
\end{equation}
with $M=\int_{{\mathbb{B}^2}}  n_0(x)\, dV_x$.
\item
If the  two conditions $\chi M>{8\pi}$ and 
\begin{equation}
\int_{{\mathbb{B}^2}}p \, n_0\,dV
< \lambda^*(M),
\end{equation}
where
\begin{equation}\label{lambdastar}
\lambda^*(M)=M\left(\sqrt{\frac{\chi M}{8\pi}}-1\right),
\end{equation}
are satisfied, then the solution $n_t$ can exist only on a finite interval $[0,T^*)$ with
\begin{equation}\label{blbound}
T^*\leq T_{bl}:=
\frac{1}{4}
\log
\left[\frac{M^2}{8\pi}
(\chi M-8\pi)
\left[\frac{\chi M^3}{8\pi}-\left(M+\int_{\mathbb{B}^2} p \, n_0\,dV\right)^2\right]^{-1}
\right]
.
\end{equation}
In particular, a smooth solution $n_t$ does not exist for $t>T_{bl}$. 
\end{enumerate}
\end{theo}
\begin{proof}
1. 
Let 
${\mathcal I}(t)=\int_{\mathbb{B}^2} p \, n_t\,dV$. 
Formally,
the derivative of ${\mathcal I}(t)$ is given by
$$
{\mathcal I}^{\,\prime}(t)=
\int_{\mathbb{B}^2} p\frac{\partial}{\partial t} n_t(x)\,dV
=
\int_{\mathbb{B}^2} p \, \Delta_{\mathbb{H}} \, n_t(x)\,dV_x-\chi \int_{\mathbb{B}^2} p \nabla_{\mathbb{H}} \cdot \left(n_t(x)\nabla_{\mathbb{H}} c_t(x)\right)\,dV_x
$$
$$
=
\int_{\mathbb{B}^2} \Delta_{\mathbb{H}} \, p\, n_t(x)\,dV_x+\chi \int_{\mathbb{B}^2} g_x(\nabla_{\mathbb{H}} \,p(x) ,n_t(x)\nabla_{\mathbb{H}} c_t(x))\,dV_x.
$$
The  second integral has been integrated by parts.
Thus we obtain 
$$
{\mathcal I}^{\,\prime}(t)= 
\int_{\mathbb{B}^2} \Delta_{\mathbb{H}} \, p \, n_t(x)\,dV_x+\chi \int_{\mathbb{B}^2} \left(\frac{2}{1-\vert x\vert^2}\right)^2 \nabla_{\mathbb{H}} p(x) \cdot  \nabla_{\mathbb{H}} c_t(x) n_t(x)\,dV_x.
$$
Here again $X  \cdot Y$ denotes the Euclidean scalar product. 
We express the hyperbolic \\ gradient with the Euclidean gradient and get
$$
{\mathcal I}^{\,\prime}(t)= 
\int_{\mathbb{B}^2} \Delta_{\mathbb{H}} \, p \, n_t(x)\,dV_x+\chi \int_{\mathbb{B}^2} \left(\frac{1-\vert x\vert^2}{2}\right)^2 \nabla_e p(x)  \cdot  \nabla_e c_t(x) n_t(x)\,dV_x.
$$
Using the Green kernel $G_{\mathbb{H}}$ of $-\Delta_{\mathbb{H}}$ to express $c_t$, we obtain for the second integral
$$
\int_{\mathbb{B}^2} \left(\frac{1-\vert x\vert^2}{2}\right)^2 \nabla_e p(x) \cdot \nabla_e c_t(x) n_t(x)\,dV_x
=
$$
$$
\int_{\mathbb{B}^2} \int_{\mathbb{B}^2}  \left(\frac{1-\vert x\vert^2}{2}\right)^2 \nabla_e p(x)  \cdot  \nabla_e^x G_{\mathbb{H}}(x,y) n_t(x)n_t(y)\,dV_xdV_y,
$$
where we recall that 
$$
G_{\mathbb{H}}(x,y)=-k_1 \log \left[ \tanh (\rho(x,y)/2)\right]
=
-\frac{k_1}{2} \log \vert T_x(y)\vert^2, \quad k_1=\frac{1}{2\pi}.
$$
To simplify our notation, we set $H(x,y)=\log \vert T_x(y)\vert^2$,
where  $T_x(y)$ is the M\"oebius transformation (see Section \ref{formule}). This leads to the following
$$
{\mathcal I}^{\,\prime}(t)=
\int_{\mathbb{B}^2} \Delta_{\mathbb{H}} \, p \, n_t(x)\,dV_x-\frac{\chi k_1}{2}
\int_{\mathbb{B}^2} \int_{\mathbb{B}^2}  \left(\frac{1-\vert x\vert^2}{2}\right)^2 \nabla_e p(x) \cdot \nabla_e^x H(x,y) \,d\mu(x,y),
$$
where $d\mu(x,y)$ is the symmetric measure given by $d\mu(x,y):=n_t(x)n_t(y)\,dV_xdV_y$. 
Note that $\mu=\mu_t$ depends on $t$.
For the second integral, we set
$$
J:=
\int_{\mathbb{B}^2} \int_{\mathbb{B}^2}  \left(\frac{1-\vert x\vert^2}{2}\right)^2 \nabla_e p(x)  \cdot  \nabla_e^x H(x,y) \,d\mu(x,y).
$$
By changing the role of $x$ and $y$ in the integral (using the symmetry of the measure and the fact that
$\vert T_x(y)\vert=\vert T_y(x)\vert$), we also have
$$
J=
\int_{\mathbb{B}^2} \int_{\mathbb{B}^2}  \left(\frac{1-\vert y\vert^2}{2}\right)^2 \nabla_e p(y)  \cdot \nabla_e^y H(x,y) \,d\mu(x,y).
$$
Hence, by symmetrization, the derivative of ${\mathcal I}(t)$ takes the next form
$$
{\mathcal I}^{\,\prime}(t)=
\int_{\mathbb{B}^2} \Delta_{\mathbb{H}} \, p \, n_t(x)\,dV_x-\frac{\chi_0}{2}
(2J),
$$
with
$\chi_0:=\frac{\chi k_1}{2}$ 
and 
$2J=
\int\int L(x,y) \,d\mu(x,y),$
where
$$
L(x,y):=
 \left(\frac{1-\vert x\vert^2}{2}\right)^2 \nabla_e p(x) \cdot  \nabla_e^x H(x,y)
+
\left(\frac{1-\vert y\vert^2}{2}\right)^2 \nabla_e p(y) \cdot \nabla_e^y H(x,y).
$$

2. {\em Computation of  $\nabla_e^x H(x,y)$}. We write the square of the Euclidean norm of the M\"oebius transformation as
$$
\vert T_x(y)\vert^2=
 \frac{\vert x-y\vert^2}
{V},
$$
with $V=V(x,y):=1-2x \cdot y+\vert x\vert^2\vert y\vert^2=\vert x-y\vert^2+ (1-\vert x\vert^2)(1-\vert y\vert^2)$.
\\
By a straightforward computation, we have
$$
\nabla_e^x \vert T_x(y)\vert^2=
 \frac{2}
{V^2}\left[ (1-\vert x\vert^2)(1-\vert y\vert^2)(x-y)
+\vert x-y\vert^2(1-\vert y\vert^2)x\right].
$$
Hence, we obtain
$$
\nabla_e^xH(x,y)\vert^2=\nabla_e^x \vert T_x(y)\vert^2/ \vert T_x(y)\vert^2
=
 \frac{2(1-\vert x\vert^2)(1-\vert y\vert^2)}
{\vert x-y\vert^2 V}(x-y)
+
 \frac{2(1-\vert y\vert^2)x}{V}.
$$
We deduce that $L$ has the following  simple form with the choice of $p$ given by \eqref{weightnatur}
$$
L(x,y)= \frac{2}{V}\left[ 1- \vert x\vert^2\vert y\vert^2\right].
$$
Note that, in the Euclidean case ${\mathbb{R}^2} $, the function $L$ is a positive constante (and $\Delta_e \vert x\vert^2$ also) and leads immediately to the blow-up equation \eqref{momentR2}. On  the Poincar\'e disk, besides the fact that $L$ is not a constant, we have to deal with a new difficulty i.e. $\inf L=0$. 
Finally, we have
$$
{\mathcal I}^{\,\prime}(t)=
\int_{\mathbb{B}^2} \Delta_{\mathbb{H}} \, p \,n_t(x)\,dV_x- {\chi_0} 
\int_{\mathbb{B}^2} \int_{\mathbb{B}^2}   \frac{ 2 \left(1- \vert x\vert^2\vert y\vert^2\right)}{V} \,d\mu(x,y).
$$
By the relation $\Delta_{\mathbb{H}} \, p=2p+2$,  the function ${\mathcal I}(t)$ satisfies the following differential equation
\begin{equation}\label{diffL}
{\mathcal I}^{\,\prime}(t)=
2\,{\mathcal I}(t)+2M- {\chi_0} 
\int_{\mathbb{B}^2}\int_{\mathbb{B}^2}   \frac{ 2 \left(1- \vert x\vert^2\vert y\vert^2\right)}{V} \,d\mu(x,y).
\end{equation}
3.
{\em Dealing with $K$}.
Let $K:= \int_{\mathbb{B}^2} \int_{\mathbb{B}^2}    \frac{\left(1- \vert x\vert^2\vert y\vert^2\right)}{V} \,d\mu(x,y).$
\\
In order to obtain an upper  bound on ${\mathcal I}^{\,\prime}(t)$ in \eqref{diffL}  i.e. a lower bound on $K$,
 we use the fact that
$$
V=V(x,y):=1-2 \,x \cdot y+\vert x\vert^2\vert y\vert^2\leq
 \left(1+\vert x\vert\vert y\vert\right)^2.
$$
Hence, we have
$$
K\geq \int_{\mathbb{B}^2} \int_{\mathbb{B}^2}    \frac{\left(1- \vert x\vert^2\vert y\vert^2\right)}{ \left(1+\vert x\vert\vert y\vert\right)^2} \,d\mu(x,y)
=\int_{\mathbb{B}^2} \int_{\mathbb{B}^2}    \frac{1- \vert x\vert\vert y\vert}{ 1+\vert x\vert\vert y\vert} \,d\mu(x,y).
$$
Then, by Cauchy-Schwartz inequality, we obtain
$$
M^2= \int_{\mathbb{B}^2} \int_{\mathbb{B}^2}  n_t(x)n_t(y)\,dV_xdV_y= \int_{\mathbb{B}^2} \int_{\mathbb{B}^2}  \,d\mu(x,y)
$$
$$
=
\int_{\mathbb{B}^2} \int_{\mathbb{B}^2}  
\left(\frac{1- \vert x\vert\vert y\vert}{ 1+\vert x\vert\vert y\vert}\right)^{1/2}
\left(\frac{1+\vert x\vert\vert y\vert}{ 1-\vert x\vert\vert y\vert}\right)^{1/2}
\,d\mu(x,y)
$$
$$
\leq 
\left(\int_{\mathbb{B}^2} \int_{\mathbb{B}^2}   \frac{1- \vert x\vert\vert y\vert}{ 1+\vert x\vert\vert y\vert} \,d\mu(x,y)\right)^{1/2}
\left(\int_{\mathbb{B}^2} \int_{\mathbb{B}^2}  \frac{1+\vert x\vert\vert y\vert}{ 1-\vert x\vert\vert y\vert} \,d\mu(x,y)\right)^{1/2}.
$$
Finally, we have
\begin{equation}\label{MK}
M^2\leq
\sqrt{K}
\left(\int_{\mathbb{B}^2} \int_{\mathbb{B}^2}  \frac{1+\vert x\vert\vert y\vert}{ 1-\vert x\vert\vert y\vert} \,d\mu(x,y)\right)^{1/2}.
\end{equation}
To obtain an upper  bound on ${\mathcal I}^{\,\prime}(t)$ in \eqref{diffL}  i.e. a lower bound on $K$ (since $-\chi_0<0$), 
it remains to prove an upper bound on 
$$Q:=\int_{\mathbb{B}^2}  \int_{\mathbb{B}^2}  \frac{1+\vert x\vert\vert y\vert}{ 1-\vert x\vert\vert y\vert} \,d\mu(x,y)$$ as follows. 
Using the fact that for any $x,y\in {\mathbb{B}^2} $
$$
1-\vert x\vert\vert y\vert
\geq 
\sqrt{1-\vert x\vert^2}
\sqrt{1-\vert y\vert^2},
$$
we deduce that
$$
Q
\leq
\int_{\mathbb{B}^2} \int_{\mathbb{B}^2}  \frac{1}{\sqrt{1-\vert x\vert^2}
\sqrt{1-\vert y\vert^2}} \,d\mu
+
\int_{\mathbb{B}^2} \int_{\mathbb{B}^2}  \frac{\vert x\vert\vert y\vert}{\sqrt{1-\vert x\vert^2}
\sqrt{1-\vert y\vert^2}} \,d\mu.
$$
Now, we can introduce again the weight $p(x)=\frac{2\vert x\vert^2}{1-\vert x\vert^2}$ using the two next formulae
$$
\sqrt{p/2}=\frac{\vert x\vert}{\sqrt{1-\vert x\vert^2}}, \quad
\sqrt{p/2+1}=\frac{1}{\sqrt{1-\vert x\vert^2}}
.
$$
This implies for $Q$ 
$$
Q\leq
\left(\int_{\mathbb{B}^2}  \left[\left(\frac{p}{2}+1\right)n_t \right]^{\frac{1}{2}} \sqrt{n_t}\,dV_x\right)^2
+
\left(\int_{\mathbb{B}^2}  \sqrt{\frac{p \,n_t}{2}} \sqrt{n_t}\,dV_x\right)^2.
$$
Again by Cauchy-Schwarz inequality applied on both integrals, we have
$$
Q\leq
M\left(\int_{\mathbb{B}^2}  {\left(\frac{p}{2}+1\right)n_t \,dV_x 
+
\int_{\mathbb{B}^2} \frac{p}{2}}{n_t}\,dV_x\right) 
=
M\left(\int_{\mathbb{B}^2}  p{n_t}\,dV_x+M\right)=M\left({\mathcal I}(t)+M\right).
$$
From \eqref{MK}, we finally get
$$
M^4\leq K Q\leq KM\left({\mathcal I}(t)+M\right),
$$
or equivalently a lower bound on $K$
\begin{equation}\label{eqk}
K\geq \frac{M^3}{{\mathcal I}(t)+M}.
\end{equation}

4. {\em A differential inequality} for ${\mathcal I}(t)$.
\\
From \eqref{diffL} and \eqref{eqk}, we deduce a differential inequality for ${\mathcal I}(t)=\int_{{\mathbb{B}^2} } p \, n_t\;dV$,
$$
{\mathcal I}^{\,\prime}(t)\leq
2({\mathcal I}(t)+M)- {\chi_0}K
\leq
2({\mathcal I}(t)+M)- {\chi_0}\frac{M^3}{{\mathcal I}(t)+M}.
$$
We set $\varphi (t)={\mathcal I}(t)+M$, the preceding inequality reads as
$$
\varphi^{\,\prime}(t)\leq 2\varphi (t) -{\chi_0}\frac{M^3}{\varphi(t)}.
$$
Thus, we have
$$
\frac{1}{2}
 \left(\varphi^2\right)^{\,\prime}(t)
\leq
2\varphi^2(t)-{\chi_0}{M^3}.
$$
Now, by setting $\psi(t)=\varphi^2(t)$, we deduce the following simple differential inequality
$$
\psi^{\,\prime}(t) - 4 \psi(t)\leq -2{\chi_0}{M^3}.
$$
By multiplying the last inequality by $e^{-4t}$, we obtain
$$
\left[e^{-4t} \psi(t)\right]^{\,\prime}
\leq
\frac{1}{2}{\chi_0}{M^3}
\left[e^{-4t}\right]^{\,\prime}.
$$
We integrate on $(0,t)$ and we get for any $t\in (0,T^*)$
$$
e^{-4t} \psi(t)-\psi(0)
\leq
\frac{1}{2}{\chi_0}{M^3}
\left[e^{-4t}-1\right].
$$
This leads to the following
\begin{equation}\label{eqpsi}
\psi(t)\leq \left[\psi(0)-\frac{1}{2}{\chi_0}{M^3}\right]e^{4t}+
\frac{1}{2}{\chi_0}{M^3}.
\end{equation}
Since  $\psi(t)=\left({\mathcal I}(t)+M\right)^2$, where ${\mathcal I}(t)=
\int_{{\mathbb{B}^2} } p \, n_t\; dV$ and $k_1=(2\pi)^{-1}$ ($\chi_0=\chi (4\pi)^{-1}$), we get \eqref{weightinequality}, i.e.
\begin{equation*}
\left(\int_{{\mathbb{B}^2} } pn_t\; dV +M\right)^2
\leq
\left(\left[\int_{{\mathbb{B}^2} } pn_0\;dV +M\right]^2-\frac{\chi}{8\pi}M^3\right)e^{4t}
+
\frac{\chi}{8\pi}M^3.
\end{equation*}

\vskip0,4cm

5. {\em Blow-up conditions.}
Now assume that the solution $n_t$ exists for any $t>0$ (i.e. $T^*=+\infty$) and assume also that
\begin{equation}\label{blowupconds1}
\left[\int_{{\mathbb{B}^2} } pn_0\;dV +M\right]^2-\frac{\chi}{8\pi}M^3
<0
\end{equation} 
holds true. We obtain a contradiction since the right-hand side term of \eqref{weightinequality} tends to 
$-\infty$ as $t\rightarrow +\infty$ and 
$\left(\int_{{\mathbb{B}^2} } pn_t\; dV +M\right)^2$  remains non-negative for any $t>0$.
\\
Then we say that a blow-up appears i.e. the solution does not exist globally in time (i.e. $T^*<+\infty$).
\begin{rem}
The inequality \eqref{blowupconds1} is equivalent to
\begin{equation}\label{momentcond2}
0\leq 
\int_{{\mathbb{B}^2} }p(x)n_0(x)\,dV_x
<
M \left(\sqrt{\frac{\chi M}{8\pi}} -1\right):={\lambda}^*(M).
\end{equation}
In fact, the inequality  \eqref{blowupconds1}  implies two conditions. 
First, the right-hand side of the equation above should be positive  
i.e. ${\chi M}>{8\pi}$
and, secondly, the $p$-moment at the initial time $t=0$ should satisfy \eqref{momentcond2}.
Note that $ M>8\pi/\chi$ is  the same condition for blow-up than the Euclidean case ${\mathbb{R}^2}$. But here, as pointed out before, it appears  the  additional  condition 
\eqref{momentcond2} for blow-up on our  weighted moment.
\end{rem}


6. {\em Estimate of the blow-up time $T_{bl}$.}
In case of blow-up  i.e. under the conditions $M>{8\pi}/{\chi }$ and \eqref{momentcond2}, 
the existence time $T^*$ of any solution $(n_t)$ of the Keller-Segel system is bounded 
by the next bound $T_{bl}$ obtained as follows. 
From \eqref{eqpsi} and 
$$
M^2\leq \psi(t)= \left({\mathcal I}(t)+M\right)^2,
$$
we have
$$
M^2\leq -\left[\frac{1}{2}{\chi_0}{M^3}-\left(M+\int_{{\mathbb{B}^2} }p \, n_0\,dV\right)^2\right]e^{4t}+
\frac{1}{2}{\chi_0}{M^3},
$$
or, equivalently for any $t<T^*$ with $\frac{1}{2}{\chi_0}=\frac{\chi}{8\pi}$,
$$
t\leq T_{bl}:=\frac{1}{4}
\log
\left[\frac{M^2}{8\pi}
(\chi M-8\pi)
\left[\frac{\chi M^3}{8\pi}-\left(M+\int_{\mathbb{B}^2}  p \, n_0\,dV\right)^2\right]^{-1}
\right]
$$
which leads to \eqref{blbound}. This completes the proof of Theorem \ref{blowupD}.
\end{proof}

\section{A priori control of the Entropy of a positive solution}
In this section, we prove lower and upper bounds of the entropy of the solution of the Keller-Segel problem \eqref{kssyst}  on the hyperbolic space. In particular, we shall prove that $n_t\log n_t\in L^1$ is locally uniformly bounded in time  under the condition $\chi M<8\pi$. The main ingredients in our proof are the use of the upper bound of the $p$-moment \eqref{weightinequality} and the logarithmic Hardy-Littlewood-Sobolev inequality on $\mathbb{B}^2$ \eqref{hlsbnls}.

\subsection{Lower bound on the Entropy of the solution of Keller-Segel system}

We prove a lower bound on the entropy of the solution of the Keller-Segel problem \eqref{kssyst} on the hyperbolic space by using the tool of relative entropy as in \cite{BDP}. 
\begin{lem}\label{lemlowerentropy}
Assume that a function $f: {\mathbb{B}^2} \rightarrow [0,+\infty)$ is such that 
$M:=\int_{{\mathbb{B}^2} } f\,dV$ and $\int_{{\mathbb{B}^2} } pf\,dV$  are finite.  Then, for all $s>0$, we have
\begin{equation}\label{lowerentropy}
\int_{{\mathbb{B}^2} } f\log f\, dV
\geq 
- \frac{1}{s}\int_{{\mathbb{B}^2} } pf\,dV
+M \log \left(\frac{M}{2\pi s}\right).
\end{equation}
Here $p=p(\rho)=2\sinh^2 (\rho/2)$ denotes the weight defined in \eqref{weightnatur}.
\end{lem}

\begin{proof}
Let $s>0$ and define 
$$
q(x):=q_s(x)=\frac{1}{2\pi s}\exp\left(-\frac{p(x)}{s}\right)
=\frac{1}{2\pi s}\exp\left(\frac{-2\sinh^2 (\rho/2)}{s}\right).
$$
 By using the volume element $dV$ in spherical coordinates, we easily see that $q_s$ 
 is a density of a probability measure on ${\mathbb{B}^2} $ with respect to  the mesure $dV$.
A lower bound on the relative entropy of $f$ with respect to $q$ is given by
$$
\int_{{\mathbb{B}^2} } f\log f\, dV-\int_{{\mathbb{B}^2} } f\log q\, dV
=\int_{{\mathbb{B}^2} } \frac{f}{q}\log \left(\frac{f}{q} \right) q\, dV
\geq \Psi\left(\int_{{\mathbb{B}^2} } \frac{f}{q} q\, dV\right)= M \log M.
$$
The last inequality is obtained by Jensen's inequality applied to the convex function
$\Psi(u)=u\log u$ and the probability measure $q dV$.
We deduce that 
$$
\int_{{\mathbb{B}^2} } f\log f\, dV
\geq 
\int_{{\mathbb{B}^2} } \left( -\frac{p}{s}-\log (2\pi s)\right) f \, dV + M \log M,
$$
which is nothing else than \eqref{lowerentropy}. 
\end{proof}

By using this general Lemma \ref{lemlowerentropy} and the estimates \eqref{weightinequality} obtained in Section \ref{blowupsect}, 
we obtain the following lower bound on the entropy of the positive solution ($n_t$) of Keller-Segel system \eqref{kssyst}.

\begin{lem}\label{lemlowerentropysol}
Let $n: [0,T)  \times {\mathbb{B}^2} \rightarrow  {\mathbb{R}^+}$ be a solution of the Keller-Segel system  \eqref{kssyst}
with $T\leq +\infty$ such that $n \in  \mathcal{C} ([0, T),L_+^1(\mathbb{B}^2, (1+p) dV)$. 
Assume that  $M:=\int_{{\mathbb{B}^2} } n_t\,dV=\int_{{\mathbb{B}^2} } n_0\,dV$ and 
$\int_{{\mathbb{B}^2} } p \, n_0\,dV$  are finite.  Then, for all $ 0<t<T$ and $s>0$ we have  the following lower bound on the entropy of the positive solution ($n_t$) of Keller-Segel system \eqref{kssyst}
\begin{equation}\label{lowerentropysol1}
\int_{{\mathbb{B}^2} } n_t\log n_t\, dV
\geq 
- \frac{1}{s}\left[ \left(\int_{{\mathbb{B}^2} } p \, n_0\,dV)\right)e^{2t}
+M (e^{2t}-1)\right]
+M \log \left(\frac{M}{2\pi s}\right).
\end{equation}
With $s=e^{2t}$, we deduce for all $t>0$,
\begin{equation}\label{lowerentropysol2}
\int_{{\mathbb{B}^2} } n_t\log n_t\, dV
\geq 
-  \int_{{\mathbb{B}^2} } p \, n_0\,dV-M
+M e^{-2t}
+M \log \left(\frac{M}{2\pi }\right)-2Mt.
\end{equation}
\end{lem}
\begin{proof}
Set  ${\mathcal I}(t):=\int_{{\mathbb{B}^2} } pn_t\, dV$, by using the estimates  \eqref{weightinequality} we have 
$$
\left({\mathcal I}(t)+M\right)^2\leq  \left(I(0)+M\right)^2e^{4t}
+ \frac{1}{2}\chi_0M^3 (1-e^{4t})
\leq
\left(I(0)+M\right)^2e^{4t}.
$$
Thus, we deduce the next upper bound on the $p$-moment ${\mathcal I}(t)$ for any $t>0$,
$$
{\mathcal I}(t)=\int_{{\mathbb{B}^2} } pn_t\, dV\leq 
\left( \int_{{\mathbb{B}^2} } pn_0\, dV \right) e^{2t}+ M(e^{2t}-1).
$$
By Lemma \ref{lemlowerentropy} with $f=n_t$, we get
$$
\int_{{\mathbb{B}^2} } n_t\log n_t\, dV
\geq 
- \frac{1}{s}\int_{{\mathbb{B}^2} } pn_t\,dV
+M \log \left(\frac{M}{2\pi s}\right),
$$
since $M:= \int_{{\mathbb{B}^2} } n_t\,dV=\int_{{\mathbb{B}^2} } n_0\,dV$ for all $t>0$. From the upper bound on ${\mathcal I}(t)$ just above, we immediately deduce   \eqref{lowerentropysol1}  and its consequence
\eqref{lowerentropysol2} when choosing $s=e^{2t}$. 
\end{proof}

\subsection{Upper bound on the Entropy of the solution of Keller-Segel system}
The aim of this section is to obtain upper bound estimates for the entropy of the positive solution ($n_t$) of Keller-Segel system \eqref{kssyst}. To do so, we prove a suitable logarithmic Hardy-Littlewood-Sobolev inequality  on  $\B^2$, involving the entropy of the function and the Green kernel on $\B^2$.

\subsubsection{Logarithmic Hardy-Littlewood-Sobolev type inequality on  $\B^2$ 
 }
We start by recalling a Hardy-Littlewood-Sobolev type inequality  proved recently on the hyperbolic space $\B^n$ by G. Lu and Q. Yang, see \cite{LuYa}.

\begin{theo}\cite{LuYa}.
Let $n\in \N^*$, $0<\lambda<n$ and  $p=p(\lambda)=\frac{2n}{2n-\lambda}$. 
For all $f,g\in L^p(\mathbb{B}^n)$, we have
\begin{equation}\label{hlsbn}
\int_{\mathbb{B}^n}\int_{\mathbb{B}^n} \frac{\vert f(x) g(y)\vert}{\left[ 2\sinh (\rho(x,y)/2)\right]^{\lambda}}
\,dV_xdV_y
\leq
C_{n,\lambda} \vert\vert f\vert\vert_{L^p({\mathbb{B}^n})} \vert\vert g\vert\vert _{L^p({\mathbb{B}^n})},
\end{equation}
where $C_{n,\lambda}$ is given by 
\begin{equation}\label{cnlam}
C_{n,\lambda} =\pi^{\lambda/2}\, \frac{\Gamma(n/2-\lambda/2)}{\Gamma(n-\lambda/2)}
\left(\frac{\Gamma(n/2)}{\Gamma(n)}\right)^{-1+\lambda/n}
\end{equation}
and $\rho(x,y)$ is the hyperbolic metric.
\end{theo}
By the same derivation argument at $\lambda=0$ used by Carlen and Loss \cite{CL}, we deduce 
the following  version of logarithmic Hardy-Littlewood-Sobolev inequality on $\B^n$.
\begin{theo}\label{loghls}
Let $n\in \N^*$. For all $f\in L^1( {\mathbb{B}^n} )$ with $f\geq 0$, we have
\begin{equation}\label{hlsbn2}
\int_{{\mathbb{B}^n} } f(x)\log f(x)\,dV_x
+\frac{n}{M}
\int_{{\mathbb{B}^n} } \int_{{\mathbb{B}^n} }
f(x)  f(y) 
\log \left[ 2\sinh (\rho(x,y)/2)\right]
\,dV_xdV_y
\end{equation}
$$
\geq
-C_{n}(M), 
$$
with $M:=\int_{{\mathbb{B}^n} } f(x)\,dV_x$ and
$$
C_{n}(M)=
-M\log M+nM \frac{d}{d\lambda} C_{n,\lambda}\vert_{\lambda=0}
$$
$$
=
-M\log M+nM\left[\frac{1}{2}\log \pi+\frac{1}{n}\log \frac{\Gamma(n/2)}{\Gamma(n)}
+\frac{1}{2} (\Psi(n)-\Psi(n/2))
\right],
$$
where $\Psi$ is the logarithmic derivative of the Gamma function.
In particular, 
$C_{2}(M):=M\log(e\pi M)$.
\end{theo}
The constants $C_{n,\lambda} $ and  $C_{n}(M)$ are sharp and exactly the same 
as in the Euclidean space ${\mathbb{R}^n} $, see \cite{LuYa} and  \cite{CL}. 
\\

Now recall that, for non-negative functions on $\R^2$, the classical logarithmic Hardy-Littlewood-Sobolev inequality can be expressed as
\begin{equation}\label{hlsr2}
\int_{{\mathbb{R}^2} } f(x)\log f(x)\,dx-\frac{4\pi}{M}
\int_{{\mathbb{R}^2} } \int_{{\mathbb{R}^2} }f(x)  f(y) G_e(x,y)\,dxdy
\geq -M\log(e\pi M), 
\end{equation}
with $M=\int_{{\mathbb{R}^2} } f(x)\,dx$ and the Euclidean Green kernel given by
$$
G_e(x,y)=-\frac{1}{2\pi}\log\vert x-y\vert.
$$
In particular, note that the classical inequality \eqref{hlsr2} involves naturally the Green kernel, which is  the important link with the Keller-Segel system on $\R^2$.
Here, because the expression of the Green kernel $G_{\mathbb{H}}$ on $\mathbb{B}^2$ is given by
$$
G_{\mathbb{H}}(x,y)=- \frac{1}{2\pi}\log \left[ \tanh (\rho(x,y)/2)\right],
$$
the inequality \eqref{hlsbn2} is not enough to be connected with the Keller-Segel system \eqref{kssyst}.
So, in the following Theorem, we state  a suitable logarithmic Hardy-Littlewood-Sobolev inequality on $ \mathbb{B}^2$, which allow us to obtain upper bound estimates for the entropy of the solution of the Keller-Segel system \eqref{kssyst}.

\begin{theo}\label{loghlsmodif}
 For all $f\in L^1(\B^2)$ with $f\geq 0$, we have
\begin{equation}\label{hlsbnls}
\int_{{\mathbb{B}^2} } f(x)\log f(x)\,dV_x
-\frac{4\pi}{M}
\int_{{\mathbb{B}^2} } \int_{{\mathbb{B}^2} }
f(x)  f(y) 
G_{\mathbb{H}}(x,y)
\,dV_xdV_y
\end{equation}
$$
\geq -K_{2}(M)-2 \int_{{\mathbb{B}^2} }  \rho(x) f(x)\,dV_x,
$$
with $M:=\int_{{\mathbb{B}^2} } f(x)\,dV_x$ and
$K_2(M):=M\log(4e\pi M)$.
\end{theo}
\begin{proof}
Let $d=\rho(x,y)$ to avoid confusion with $\rho(x):=\rho(x,0)$.
From Theorem \ref{loghls} with $n=2$, we have
$$
\int_{{\mathbb{B}^2} }  f(x)\log f(x)\,dV_x
+ 2M\log 2
+\frac{2}{M}
\int_{{\mathbb{B}^2} }  \int_{{\mathbb{B}^2} } f(x)  f(y) 
\log \left[\sinh (d/2)\right]
\,dV_xdV_y
\geq
-C_{2}(M) 
$$
because
$$
\frac{2}{M}
\int_{{\mathbb{B}^2} }  \int_{{\mathbb{B}^2} } 
f(x)  f(y) 
\log 2 \,dV_xdV_y=2M\log 2.
$$
This can be rewritten as
\begin{equation}\label{ineqtanh}
\int_{{\mathbb{B}^2} } f(x)\log f(x)\,dV_x
+\frac{2}{M}
\int_{{\mathbb{B}^2} }  \int_{{\mathbb{B}^2} } 
f(x)  f(y) 
\log \left[ \tanh (d/2)\right]\,dV_xdV_y
\end{equation}
$$
+
\frac{2}{M} \int_{{\mathbb{B}^2} }  \int_{{\mathbb{B}^2} } f(x)  f(y) 
\log \left[ \cosh (d/2)\right]
\,dV_xdV_y
\geq
-C_{n}(M)-2M\log 2.
$$
Now by using the triangle inequality for the metric $\rho$, we have the following
$$
\log\left[\cosh(\rho(x,y)/2\right]\leq  \rho(x,y)/2
\leq
\rho(x,0)/2+ \rho(y,0)/2=\frac{1}{2}\left[\rho(x)+ \rho(y)\right],
$$
by multiplying the last inequality by $\frac{2}{M}f(x)f(y)\geq 0$, integrating over 
${\mathbb{B}^2}\times {\mathbb{B}^2}$ and using Fubini's Theorem,
we obtain 
\begin{equation}\label{ff1}
\frac{2}{M} \int_{{\mathbb{B}^2} }  \int_{{\mathbb{B}^2} } f(x)  f(y) 
 \log\left[\cosh(\rho(x,y)/2\right]
\,dV_xdV_y
\end{equation}
$$
\leq 
\frac{1}{M} \int_{{\mathbb{B}^2} }  \int_{{\mathbb{B}^2} } f(x)  f(y) 
\left[\rho(x)+\rho(y)\right]
\,dV_xdV_y
$$
$$
=\frac{2}{M} 
\left( \int_{{\mathbb{B}^2} }   f(y)\,dV_y \right)\left(\int_{{\mathbb{B}^2} } \rho(x) f(x)\,dV_x\right)
=2 \int_{{\mathbb{B}^2} }  \rho(x) f(x)\,dV_x.
$$
Then the inequality \eqref{hlsbnls} follows from inequalities \eqref{ineqtanh} and \eqref{ff1} above.
This completes the proof.   
\end{proof}

\subsubsection{Entropy upper bound}
The next results are used for the study of the entropy upper bound of the solution of the Keller-Segel system \eqref{kssyst}.
Here, we consider the Lyapunov functional defined by
$$
F[n_t]=\int_{{\mathbb{B}^2} } n_t \,\log n_t \; dV-\frac{\chi}{2}\int_{{\mathbb{B}^2} } n_t \, c\; dV
$$
similar to the Euclidean one given in \cite{BDP}.

From now we shall perform a priori estimates in the setting of smooth enough solutions, the one constructed in Theorem \ref{theol2}. In particular, for such a solution we can define $T^*$ as the maximal existence time which is characterized by if $T^*<+\infty$, then $ \sup_{t\in[0, T^*[} \|n(t)\|_{{L^q( \mathbb{B}^2)}} = +\infty$.

\begin{pro}\label{decaylyapu}
Let $n_t$ be a solution of the Keller-Segel system \eqref{kssyst} as above. Assume that $F[n_0]$ is finite,
then we have 
$$
\frac{\partial}{\partial t} F[n_t]=- \int_{{\mathbb{B}^2} } n_t \vert \nabla_{\mathbb{H}} \log n_t -\chi \nabla_{\mathbb{H}} c\vert^2\; dV \leq 0,
$$
for all $t\in (0,T^*).$
In particular,
$F[n_t]\leq F[n_0]$ for all $0<t<T^*$.
\end{pro}
\begin{proof} The proof is similar to the Euclidean one and relies on the existence of an explicit expression of the
Green kernel on $\mathbb{B}^2.$
\end{proof}
Next, by using the inequality \eqref{weightinequality}  we obtain an upper bound on the $p$-moment of the solution of the Keller-Segel system \eqref{kssyst}. 
\begin{pro} \label{itbound}
Let $n_t$ be a solution of the Keller-Segel system \eqref{kssyst} as above. 
Then, for all $0<t<T^*$, we have the following upper bound of the $p$-moment
\begin{equation}\label{lambdastarpnbound}
\int_{{\mathbb{B}^2} } p \, n_t\; dV\leq C_+(p,n_0)e^{2t}+
\lambda^*(M),
\end{equation}
with
\begin{equation}\label{cplus}
C_+(p,n_0)=\left(\int_{{\mathbb{B}^2} } p \, n_0\;dV- \lambda^*(M)\right)_+^{1/2}
\left(\int_{{\mathbb{B}^2} } p\, n_0\;dV+ M+ M\sqrt{\frac{\chi M}{8\pi}}\right)^{1/2},
\end{equation}
where $\lambda^*(M)$ is given by \eqref{lambdastar} and $M=\int_{{\mathbb{B}^2} } n_0\;dV$.
\end{pro}

\begin{proof}
We recall here \eqref{weightinequality}
$$ 
\left(\int_{{\mathbb{B}^2} } p\,n_t\; dV +M\right)^2
\leq
\left(\left[\int_{{\mathbb{B}^2} } p\,n_0\;dV +M\right]^2-\frac{\chi}{8\pi}M^3\right)e^{4t}
+
\frac{\chi}{8\pi}M^3,
$$ 
valid for all $t\in (0,T^*)$ with $T^*\leq +\infty$.
By taking the square root on both sides and using the inequality 
$\sqrt{a+b}\leq \sqrt{a}+\sqrt{b}, a,b\geq 0$, we obtain
$$
 \int_{{\mathbb{B}^2} } p\, n_t\; dV  
\leq
\left(\left[\int_{{\mathbb{B}^2} } p \, n_0\;dV +M\right]^2-\left(M\sqrt{\frac{\chi M}{8\pi}}\right)^2\right)^{1/2}_+e^{2t}
+
\sqrt{\frac{\chi}{8\pi}M^3}-M.
$$
Hence, we have
$$
\int_{{\mathbb{B}^2} } p \, n_t\; dV  
\leq
\left(\int_{{\mathbb{B}^2} } p \, n_0\;dV -\lambda^*(M)\right)^{1/2}_+
\left(\int_{{\mathbb{B}^2} } p \, n_0\;dV +M+ M\sqrt{\frac{\chi M}{8\pi}}\right)^{1/2}
e^{2t}
+
\lambda^*(M).
$$
So, we obtain
$$
\int_{{\mathbb{B}^2} } p \, n_t\; dV\leq C_+(p,n_0)e^{2t}+
\lambda^*(M),
$$
where $C_+(p,n_0)$ and $\lambda^*(M)$ are given respectively by \eqref {cplus} and \eqref{lambdastar}.
This concludes the proof of the proposition.
\end{proof}
\begin{rem}
Note that if $\int_{{\mathbb{B}^2} } pn_0\;dV- \lambda^*(M) \leq 0$, it implies that $\lambda^*(M)\geq 0$, hence $\chi M \geq 8\pi$. Under the condition of blow-up or critical case  we get $ C_+(p,n_0)=0$ and  so, we have  the following  uniform bounds in time of the $p$-moment
$$
\int_{{\mathbb{B}^2} } p \, n_t\; dV\leq
\lambda^*(M),
$$
for any $t\in (0,T^*)$. 
\end{rem}

%
%
Now, we shall control  the $\rho$-moment of the solution $n_t$ of the Keller-Segel system \eqref{kssyst} $\int_{{\mathbb{B}^2} } \rho \, n_t \; dV$, which appears in the  lower bound of  our logarithmic Hardy-Littlewood-Sobolev inequality on ${\mathbb{B}^2}$ \eqref{hlsbnls} with $f=n_t$.
\begin{pro}\label{pcontrolrho}
 For any solution $n_t$ of the Keller-Segel system \eqref{kssyst} with $T^*\leq +\infty$, we have
 $$
 0\leq \int_{{\mathbb{B}^2} } \rho \, n_t \; dV\leq
 K_+ + 2Mt, \quad 0<t<T^*,
 $$
 with $M=\int_{{\mathbb{B}^2} }  n_0\; dV$, $\rho=\rho(x,0)$ is  the distance from $x\in \B^2$ to the center $0$ of $\B^2$
 and
 $$
 K_+= 2M\log\left(
\sup\left(2\frac{C_+}{2M},2\sqrt{\frac{\lambda_+^*(M)}{2M}}+1\right)\right),
 $$
with $C_+$ given by \eqref{cplus} of Proposition \ref{itbound} .
\end{pro}

\begin{proof} We apply  Jensen's inequality with $\Phi(u)=\sinh(u)$ as convex function 
and $d\mu=\frac{n_t}{M}\; dV$ as probability measure.  For any $0<t<T^*\leq +\infty$, we get
$$
\int_{{\mathbb{B}^2} } \rho \, n_t\; dV
\leq 2M\sinh^{-1}\left( 
\frac{1}{\sqrt{2M}}\left(
\int_{{\mathbb{B}^2} } p \, n_t\; dV\right)^{1/2}
\right).
$$
By using the inequality $\sinh^{-1}(u)\leq \log (2u+1), u\geq 0$ and Proposition \ref{itbound}, we obtain 
$$
\int_{{\mathbb{B}^2} } \rho \, n_t\; dV
\leq 2M
\sinh^{-1}\left(\sqrt{\frac{C_+}{2M}}e^t+
\sqrt{\frac{\lambda_+^*(M)}{2M}}\right)
$$
$$
\leq
2M\log
\left(2\sqrt{\frac{C_+}{2M}}e^t+
2\sqrt{\frac{\lambda_+^*(M)}{2M}}+1
\right)
\leq   K_+ +2Mt,
$$
with $K_+=2M\log\left(
\sup(2\sqrt{\frac{C_+}{2M}},2\sqrt{\frac{\lambda_+^*(M)}{2M}}+1)\right)$.
The proof is completed. 
\end{proof}

The following Lemma is nothing else than the Lemma 8 of \cite{BDP} written  in a general measure theory context.
\begin{lem} \label{uvertloguvert}
\begin{enumerate}
\item
Let $(E,\nu)$ be a measure space. Assume that $u,q:E\rightarrow [0,+\infty[$ satisfy the following conditions
$M:=\int u\; d\nu$,  $\int u\log u\; d\nu<$ and $\int (-\log q) u\; d\nu<\infty$ are finite. 
The measure $qdV$ is a density of probability i.e.
$\int q\; d\nu=1$.
Then for $v=u1_{0\leq u\leq 1}$, we have 
$$
\int u\vert \log u\vert \; d\nu
\leq
\int u\log u\; d\nu+\frac{2}{e}+2\int (-\log q) v\; d\nu.
$$
\item
Let $E=\B^2$ and $\nu=dV$ the hyperbolic measure. Let $p$ the weight given by \eqref{weightnatur}. Then,
we have for any $s>0$
\begin{equation}\label{weightp}
\int_{{\mathbb{B}^2} } u\vert \log u\vert \; d\nu
\leq
\int_{{\mathbb{B}^2} }  u\log u\; d\nu+\frac{2}{e}+ 2M\log(2\pi s)
+\frac{1}{s}\int_{{\mathbb{B}^2}} p u\; d\nu.
\end{equation}
\end{enumerate}
\end{lem}

Finally, we prove the upper bound of the entropy of a positive solution. Note that, we are  able to prove that $n_t\log n_t\in L^1$ is locally uniformly bounded in time  under the condition $\chi M<8\pi$.


\begin{pro}\label{entropybound}
Let $ n_t $ be a positive solution of Keller-Segel system \eqref{kssyst} and $T^*$ the maximal  existence time of the solution. Assume that 
$\int_{{\mathbb{B}^2} } p \, n_0\,dV$ 
and $F[n_0]$ are finite.
 \begin{enumerate}
\item
For any $0<t\leq T< T^*$ and  $\chi M<8\pi$, we have the following estimate
\begin{equation}\label{chimpositive}
\int_{{\mathbb{B}^2} } n_t\log n_t\,dV
\leq 
\left(1-\frac{\chi M}{8\pi}\right)^{-1}C(n_0,T),
\end{equation}
with
$$
C(n_0,T):=F[n_0] +\frac{\chi M}{8\pi}\left( 
K_{2}(M) +2K_+ + 4MT\right),
$$
where $K_{2}(M)$ and $K_+$ are defined respectively in Proposition \ref{pcontrolrho} and Theorem \ref{loghlsmodif}.
\item
Moreover, for all $0<t \leq T <T^*$, $s>0$ and  $\chi M<8\pi$, we have
\begin{equation}\label{logabs1}
\int_{{\mathbb{B}^2} } n_t\vert \log n_t\vert \,dV
\leq 
\left(1-\frac{\chi M}{8\pi}\right)^{-1}C(n_0,T)
+
\frac{2}{e}+ 2M\log(2\pi s)
+\frac{1}{s}\int_{\mathbb{B}^2} p \, n_t\; dV
\end{equation}
and its consequence
\begin{multline}\label{logabs2}
\int_{{\mathbb{B}^2} } n_t\vert \log n_t\vert \,dV
\leq 
\left(1-\frac{\chi M}{8\pi}\right)^{-1}C(n_0,T)
+
\frac{2}{e} \\+ 2M\log(2\pi s)
+s^{-1}C_+e^{2T}+
s^{-1}\lambda^*(M),
\end{multline}
with $C_+=C_+(p,n_0)$ given by \eqref{cplus}.
\end{enumerate}
\end{pro}

\begin{proof}
1) The proof follows the same lines as in Lemma 7 (upper bound) of \cite{BDP}. We provide the details.

{\em Step 1: Use of the  decay of the functional $F[n_t]$.}
\\
From Proposition \ref{decaylyapu}, we have $F[n_t]\leq F[n_0]$ for all $0<t\leq T< T^*$ i.e.
$$
F[n_t]=\int_{{\mathbb{B}^2} } n_t\log n_t\; dV-\frac{\chi}{2}\int_{{\mathbb{B}^2} } n_tc_t\; dV
\leq
F[n_0]
$$
which can be written as
$$
\int_{{\mathbb{B}^2} } n_t\log n_t\; dV-\frac{\chi}{2}<n_t,(-\Delta_{\mathbb{H}})^{-1}n_t> \,
\leq 
F[n_0],
$$
with $<f,g>=\int_{{\mathbb{B}^2} } fg\,dV$. Then we deduce the following inequality
\begin{equation}\label{fourpi}
\left(1-\frac{\chi M}{8\pi}\right)\int_{{\mathbb{B}^2} } n_t\log n_t\; dV
+
\frac{\chi M}{8\pi}\left( 
\int_{{\mathbb{B}^2} } n_t\log n_t\; dV-\frac{4\pi}{M}<n_t,(-\Delta_{\mathbb{H}})^{-1}n_t> 
\right)
\leq
F[n_0].
\end{equation}
{\em Step 2:  Use Logarithmic of Hardy-Littlewood-Sobolev inequality on $\mathbb{B}^2.$}
From logarithmic Hardy-Littlewood-Sobolev inequality on the hyperbolic space proved in Theorem \ref{loghlsmodif} with $f=n_t$, 
we have
$$
\int_{{\mathbb{B}^2} } n_t\log n_t\,dV
-\frac{4\pi}{M}
\int_{{\mathbb{B}^2} } \int_{{\mathbb{B}^2} }n_t(x)  n_t(y) G_{\mathbb{H}}(x,y)\,dV_xdV_y
\geq
-K_{2}(M) -2\int_{{\mathbb{B}^2} } \rho \, n_t\,dV.
$$
This can be written as 
\begin{equation}\label{hlshyperbolic}
\int_{{\mathbb{B}^2} } n_t\log n_t\,dV
-\frac{4\pi}{M}
<n_t,(-\Delta_{\mathbb{H}})^{-1}n_t>
\geq
-K_{2}(M) -2\int_{{\mathbb{B}^2} } \rho \, n_t\,dV,
\end{equation}
Combining \eqref{fourpi} and \eqref{hlshyperbolic}, we obtain
$$
\left(1-\frac{\chi M}{8\pi}\right)\int_{{\mathbb{B}^2} } n_t\log n_t\; dV
\leq
F[n_0] +\frac{\chi M}{8\pi}\left( 
K_{2}(M) +2\int_{{\mathbb{B}^2} } \rho \, n_t\,dV\right).
$$
{\em Step 3. Upper bound on the $\rho$-moment}.
By Proposition \ref{pcontrolrho}, we have
$$
 0\leq \int_{{\mathbb{B}^2} } \rho \, n_t \; dV\leq
 K_+ + 2Mt, \quad 0<t\leq T< T^*.
$$
This implies that
\begin{equation}\label{criticalineqentropy}
\left(1-\frac{\chi M}{8\pi}\right)\int_{{\mathbb{B}^2} } n_t\log n_t\; dV
\leq
F[n_0] +\frac{\chi M}{8\pi}\left( 
K_{2}(M) +2K_+ + 4Mt\right)
\end{equation}
from which we conclude that
$$
\int_{{\mathbb{B}^2} } n_t\log n_t\; dV
\leq
\left(1-\frac{\chi M}{8\pi}\right)^{-1}
\left[
F[n_0] +\frac{\chi M}{8\pi}\left( 
K_{2}(M) +2K_+ + 4MT\right)
\right]
$$
for $0<t\leq T< T^*$ and $\chi M<8\pi$.
\vskip0.3cm
2) To prove \eqref{logabs1}, we apply 2) of Lemma \ref{uvertloguvert} with $u=n_t$ and the first part of this proposition.
To deduce \eqref{logabs2}, we use the upper bound of $p$-moment \eqref{lambdastarpnbound} of Proposition \ref{itbound}.
Now, the proof  is  completed.
\end{proof}
\section{$L^q$-bounds on solutions}
The main argument of the proof of the $L^q$-bounds of the solution $(n_t)$ under the conditions $n_0\in L_+^1$
and $n_0\log n_0\in L^1$ is the control of 
$M_t(K):=\int_{{\mathbb{B}^2} } (n_t-K)_+\,dV$  for $K$ large enough.
More precisely, we have the next result.

\begin{pro}\label{mtkisundercontrol}
Under the assumptions of Proposition \ref{entropybound}, for any positive solution $n_t$  of Keller-Segel system \eqref{kssyst}, we have for all $T< T^*$ the following
\begin{equation}\label{mkcontrol}
\sup_{0\leq t \leq T} M_t(K)\leq \frac{1}{\log K} C(T)
\end{equation}
for a positive continuous non decreasing function $C(T)$ on $[0,+\infty)$
\end{pro}

\begin{proof}
Let $T<T^*$. We have, for any $K>1$ and $0\leq t \leq T$, the following
$$
M_t(K)=\int_{{\mathbb{B}^2} } (n_t-K)_+\,dV
\leq
\int_{n_t\geq K} n_t\,dV
\leq
\frac{1}{\log K}
\int_{n_t\geq K} n_t\log n_t\,dV.
$$
So, by applying \eqref{logabs2} of  Proposition \ref{entropybound}, we have
$$
M_t(K)\leq
\frac{1}{\log K}
\int_{{\mathbb{B}^2}  } n_t\vert \log n_t \vert \,dV
\leq 
\frac{1}{\log K} C(T).
$$
The function $C(T)$ is the bound 
of the inequality \eqref{logabs2} (for any fixed $s>0$).
\end{proof}
Now, we adapt Proposition 3.3 of \cite{BDP} to control  the $L^q$-bound of the solution $(n_t)$.

\begin{theo}\label{theoLqApriori}
Assume that $n_0\in L^1_+(\B^2, (1+p)dV)$ and $n_0\log n_0\in L^1(\B^2,dV)$ with $\chi M< 8\pi$ 
and the weight $p$ given by \eqref{weightnatur}. Assume in addition that $n_0\in L^q(\B^2)$ for some $1<q\leq2$.
Then the solution $n_t$  of the Keller-Segel system \eqref{kssyst}
  satisfies  the following estimate
$$
\sup_{0\leq t \leq T} \vert\vert n_t\vert\vert_{L^q( \mathbb{B}^2)} \leq N_q(T)<+\infty,
$$
where $T<T^*$ and $N_q(T)$ is a continuous function on $[0, +\infty)$.
\end{theo}

\begin{proof} 
We mention only the main steps.
The first idea is to prove a differential inequality of the form
\begin{equation}\label{gronwall} 
\phi^{\prime}(t)\leq c_1 \phi(t)+c_2
\end{equation}
for $\phi(t):=\vert\vert (n_t-K)_+\vert\vert_{L^q( \mathbb{B}^2)}^q$ for $K$ large enough ($c_1>0$)
and  apply Gronwall lemma from which we deduce that
$$
\phi(t)\leq \left(\phi(0)+\frac{c_2}{c_1}\right)e^{c_1t}-\frac{c_2}{c_1}
\leq 
\left(\phi(0)+\frac{c_2}{c_1}\right)e^{c_1T}-\frac{c_2}{c_1}, \quad 0\leq t \leq T <T^*.
$$
{\em Step 1. Reduction to $\phi(t)$ estimate.}
From theoretical measure theory considerations, we easily obtain  
$$
\vert\vert n_t \vert\vert_{L^q( \mathbb{B}^2)}^q=
\int_{n_t<\lambda K} n_t^q +
\int_{n_t\geq \lambda K} n_t^q 
\leq
(\lambda K)^{q-1}M
+
\int_{n_t\geq \lambda K} n_t^q
$$
for all $K,\lambda>1$ and $0\leq t \leq T <T^*$. So, by using $n_t^q\leq \left( \frac{\lambda}{\lambda-1}\right)^{q-1}(n_t-K)^q$ when $n_t\geq \lambda K$, we get
$$
\vert\vert n_t \vert\vert_{L^q( \mathbb{B}^2)}^q\leq
(\lambda K)^{q-1}M
+
\left( \frac{\lambda}{\lambda-1}\right)^{q-1}\int_{{\mathbb{B}^2} } (n_t-K)_+^q\,dV.
$$
{\em Step 2.} To bound $\vert\vert n_t \vert\vert_{L^q( \mathbb{B}^2)}^q$, 
we only have to deal with the second term by proving \eqref{gronwall}. We set $\phi(t):=\vert\vert (n_t-K)_+\vert\vert_{L^q( \mathbb{B}^2)}^q$.
We use multipliers method to estimate $\phi^{\prime}(t)$. After integrations by parts, we get
\begin{equation}\label{phiprime} 
\phi^{\prime}(t)=\frac{-4(q-1)}{q}\int_{{\mathbb{B}^2} } \vert \nabla_{\mathbb{H}} \left( (n_t-K)_+^{q/2}\right)\vert^2\,dV
+(2q-1)\chi K
\int_{{\mathbb{B}^2} } (n_t-K)_+^q\,dV
$$
$$
+q\chi K^2
\int_{{\mathbb{B}^2} } (n_t-K)_+^{q-1}\,dV
+
(q-1)\chi \int_{{\mathbb{B}^2} } (n_t-K)_+^{q+1}\,dV.
\end{equation}
As shown in \cite{BDP}, we have
$$
\int_{{\mathbb{B}^2} } (n_t-K)_+^{q-1}\,dV
\leq
\frac{M}{K}+\int_{{\mathbb{B}^2} } (n_t-K)_+^{q}\,dV.
$$
Thus, we get
$$
\phi^{\prime}(t)
\leq
\left[
(2q-1)\chi K
+q\chi K^2
\right]
\phi(t)
+q\chi KM
$$
$$
+\left[
\frac{-4(q-1)}{q}\int_{{\mathbb{B}^2} } \vert \nabla_{\mathbb{H}} \left((n_t-K)_+^{q/2}\right)\vert^2\,dV
+(q-1)\chi \int_{{\mathbb{B}^2} } (n_t-K)_+^{q+1}\,dV
\right].
$$
{\em Step 3. Use of Gagliardo-Nirenberg inequality}.
In order to show \eqref{gronwall}, it remains to show 
that
$$
R(n_t):=
\frac{-4(q-1)}{q}\int_{{\mathbb{B}^2} } \vert \nabla_{\mathbb{H}} \left( (n_t-K)_+^{q/2}\right)\vert^2\,dV
+(q-1)\chi \int_{{\mathbb{B}^2} } (n_t-K)_+^{q+1}\,dV
\leq 0
$$ for $K$ large enough and uniformly for any  $0 \leq t \leq T <T^*$. For that purpose, we use Gagliardo-Nirenberg inequality valid on the hyperbolic space  (see \cite{MugTal}, \cite{Varopoulos2} , \cite{Hebey} ), for every $1<q\leq2$  we have
$$
\int_{{\mathbb{B}^2} } \vert v\vert^{2(1+1/q)}\, dV
\leq
{\kappa}_p
\left(\int_{{\mathbb{B}^2} } \vert \nabla_{\mathbb{H}} v\vert^{2}\, dV\right)
\left(\int_{{\mathbb{B}^2} } \vert v\vert^{2/q}\, dV\right)
.
$$
This inequality applied to $v= (n_t-K)_+^{q/2}$ leads to
$$
\int_{{\mathbb{B}^2} } (n_t-K)_+^{q+1}\,dV
\leq
{\kappa}_q
\left(\int_{{\mathbb{B}^2} } \vert \nabla_{\mathbb{H}} \left((n_t-K)_+^{q/2}\right)\vert^2\,dV\right)
M_t(K),
$$
with $M_t(K)=\int_{{\mathbb{B}^2} } (n_t-K)_+\,dV$. Hence,
$$
R(n_t)\leq
\left[ \frac{-4(q-1)}{q}+(q-1)\chi {\kappa}_q M_t(K)
\right]
\left(\int_{{\mathbb{B}^2} } \vert \nabla_{\mathbb{H}} \left( (n_t-K)_+^{q/2}\right)\vert^2\,dV\right),
$$
and by using the inequality \eqref{mkcontrol} of Proposition \ref{mtkisundercontrol}, we have
$$
R(n_t)\leq
\left[ \frac{-4(q-1)}{q}+\frac{1}{\log K} (q-1)\chi {\kappa}_q  C(T)
\right]
\left(\int_{{\mathbb{B}^2} } \vert \nabla_{\mathbb{H}} \left( (n_t-K)_+^{q/2}\right) \vert^2\,dV\right),
$$
for all $0 \leq t \leq T <T^*$.
Thus there exists  $K=K(T)>1$ such that 
$$
\frac{-4(q-1)}{q}+\frac{1}{\log K} (q-1)\chi {\kappa}_q  C(T)\leq 0.
$$
This implies the inequality \eqref{gronwall} with
$$
c_1=c_1(T)= (2q-1)\chi K
+q\chi K^2,   \qquad c_2=c_2(T)=q\chi KM.
$$
The proof is now completed.
\end{proof}

{\it Remark.} Note that that the explicit value of the constant $\kappa_q>0$  in  Gagliardo-Nirenberg inequality 
has no importance in our estimates as in \cite{BDP}.
\\

\section{Global well-posedness} \label{Global well-posedness}
Next, we prove global existence in time of the solution under the  subcritical condition $\chi M<8\pi$. 
Let be $p$ the weight defined as in \eqref{weightnatur}.

\begin{theo}\label{theoglob}
For every $ n_0 \in L_+^1(\mathbb{B}^2)$,  with  $I_0= \int_{\mathbb{B}^2} p \, n_0\,dV <\infty$ and  $\chi M<8\pi$ 
we have global well-posedness on $X_{T,q} \cap \mathcal{C} (\mathbb{R}^{+},L_+^1(\mathbb{B}^2, (1+p) dV)$ 
for every $T>0$ of the Keller-Segel system  \eqref{kssyst}.
\end{theo}

\begin{proof}
By  Theorem \ref{theol1} if $n_0 \in L^1(\mathbb{B}^2)$ there exists $T>0$ such that there is a unique solution 
of the Keller-Segel system  \eqref{kssyst} $n_t \in X_{T,q} \cap \mathcal{C} ([0,T],L^1(\mathbb{B}^2))$, 
where $X_{T,q}=\{n_t \, |\sup_{t\in [0,T]} t^{(1-\frac{1}{q})}  \| n(t)\|_{L^q( \mathbb{B}^2)} < + \infty \},$ with $\frac{4}{3}<q<2$. 
Then there exists $T_1 \in ]0,T]$ such that $n(T_1) \in L^q(\mathbb{B}^2)\cap L_+^1(\mathbb{B}^2)$ for $\frac{4}{3}<q<2$.
We can take $n(T_1)$ as initial data and we use  Theorem \ref{theol2} to continue the solution. 
From Theorem  \ref{theol2}, 
we obtain a maximal solution on $[T_1, T^*[$ such that $n_t \in \mathcal{C} ([T_1, T^*[, L^q( \mathbb{B}^2))$.
Moreover if $T^*<+\infty$, we must have $ \sup_{t\in[T_1, T^*[} \|n(t)\|_{{L^q( \mathbb{B}^2)}} = +\infty$.
Let us now assume that   $T^*<+\infty$,  it suffices  to prove that
$ \sup_{t\in[T_1, T^*[} \|n(t)\|_{{L^q( \mathbb{B}^2)}} < +\infty$ to get a contradiction and thus obtain a global solution.
Since  $n(T_1) \in L^q(\mathbb{B}^2)\cap L^1_+(\mathbb{B}^2, (1+p) dV)$ for $\frac{4}{3}<q<2$, 
$\int_{\mathbb{B}^2} n(T_1) \log{n(T_1)} dV$  is well-defined.
Consequently, we can use Theorem \ref{theoLqApriori} to obtain that
$$ \sup_{0\leq t <  T^*} \vert\vert n_t\vert\vert_{L^q( \mathbb{B}^2)} \leq N_q(T^*)<+\infty.$$
 This ends the proof.
\end{proof}

\section{A priori estimates of the entropy and $L^q$-norms of solutions}
In this section, we shall derive some interesting a priori estimates on the entropy and $L^q$-norms of the solutions of the Keller-Segel system \eqref{kssyst}.
To do so, we  use  the classical multiplier's method and some functional inequalities, as for example the Poincar\'e-Sobolev  inequality by Mugelli-Talenti \cite {MugTal} to have \eqref{weakineqent} and also Log-Sobolev inequality by Beckner \cite{Beck} to prove \eqref{strongineqent}. 
Let $I(n)=\int_{{\mathbb{B}^2} } \frac{\vert \nabla_{\mathbb{H}} n \vert^2}{n}\, dV$ be the Fisher information and the mass
$M=\int_{{\mathbb{B}^2} } n_t\,dV=\int_{{\mathbb{B}^2} } n_0\,dV$,
where $(n_t)_{0<t<T^*}$ is the non-negative solution of the Keller-Segel system \eqref{kssyst}.

\begin{pro}\label{Entropy decay}
Let $(n_t)_{0<t<T^*}$ be a non-negative solution of the Keller-Segel system \eqref{kssyst}, where $T^*\leq +\infty$ is the maximal existence time  of the solution. Then
\begin{enumerate}
\item
For all $t\in (0,T^*)$, we have
\begin{equation}\label{weakdiffineqent}
\frac{\partial}{\partial t} {\rm Ent}(n_t)\leq \left[ -1+\frac{\chi M}{4\pi}\right] I(n_t)
-\frac{\chi M^2}{4\pi}.
\end{equation}
\item
If ${\chi} M\leq 4\pi$ then  $t\rightarrow {\rm Ent}(n_t)$ is non-increasing on $(0,T^*)$ and we have
\begin{equation}\label{weakineqent}
{\rm Ent}(n_t)\leq -\frac{\chi M^2}{4\pi}t+{\rm Ent}(n_0).
\end{equation}
for all $t\in (0,T^*)$.
In particular if $T^*=+\infty$, 
$\lim_{t\rightarrow +\infty} {\rm Ent}(n_t)=-\infty$.
\item
Under  the condition ${\chi} M<4\pi$,  we  have 
\begin{equation}\label{strongineqent}
{\rm Ent}(n_t)\leq -\frac{\chi M^2}{4\pi}t
\end{equation}
$$
- M\log\left[
\exp\left(-\frac{1}{M}{\rm Ent}(n_0)\right)
+\frac{4\pi e}{ M}
\left(\frac{4\pi}{\chi M} -1\right)\left(1-e^{-\frac{\chi M t}{4\pi}}\right)
\right].
$$
for all $t\in (0,T^*)$.
\end{enumerate}
\end{pro}

{\it Remark.} It can be shown that the inequality \eqref{strongineqent} is stronger than \eqref{weakineqent}. But the asymptotic behavior of the upper bound  is the same i.e. $-\frac{\chi M^2}{4\pi}t$ as $t\rightarrow +\infty$ (if $T^*=+\infty$).

\begin{proof} 1).
By multiplying the first equation in the Keller-Segel system \eqref{kssyst} by $\log n_t$ and integrating by parts, we get
$$
\frac{\partial}{\partial t} {\rm Ent}(n_t)=
-I(n_t)+\chi \int_{{\mathbb{B}^2} } n_t^2\,dV.
$$
By using the following Poincar\'e-Sobolev  inequality of Mugelli-Talenti, with $u=n_t$,
\begin{equation}\label{mug-tal}
\left( \int_{{\mathbb{B}^2} } \vert u\vert \,dV\right)^2+4\pi\int_{{\mathbb{B}^2} } \vert u\vert^2 \,dV
\leq
\left(\int_{{\mathbb{B}^2} } \vert \nabla_{\mathbb{H}} u\vert \,dV\right)^2,
\end{equation}
 and by  the Cauchy-Schwarz inequality, we deduce that
$$
M^2+ 4\pi\int_{{\mathbb{B}^2} }  n_t^2 \,dV
\leq
\left(\int_{{\mathbb{B}^2} } \vert \nabla_{\mathbb{H}} n_t\vert \,dV\right)^2
\leq
\left(\int_{{\mathbb{B}^2} }    n_t \,dV\right) I(n_t)
=
M I(n_t).
$$
Then, we have
$$
\frac{\partial}{\partial t} {\rm Ent}(n_t)
\leq
-I(n_t)+\frac{\chi}{4\pi}  \left[ 4\pi \int_{{\mathbb{B}^2} } n_t^2\,dV\right]
\leq
-I(n_t)+\frac{\chi}{4\pi}  \left(
M I(n_t) -M^2
\right).
$$
We conclude that \eqref{weakdiffineqent} holds true.
\\

2) 
Since $I(n_t)\geq0$ and $\chi M\leq 4\pi$, from \eqref{weakdiffineqent} we obtain 
$
\frac{\partial}{\partial t} {\rm Ent}(n_t)
\leq
-\frac{\chi M^2}{4\pi}
$
 Then the  conclusion  is obtained by integration.
\\

3) W. Beckner proved in \cite{Beck} that the following Log-Sobolev inequality holds on the hyperbolic space
with $\vert\vert u\vert\vert_{L^2(\mathbb{B}^2)}=1$,

$$
\int_{\mathbb{B}^2} \vert u\vert^2\log \vert u\vert\, dV
\leq
\frac{1}{2}\log\left(
\frac{1}{\pi e}
\int_{\mathbb{B}^2} \vert \nabla_{\mathbb{H}} u\vert^2\, dV\right).
$$
We set $u=\sqrt{n_t/M}$ and get 
$$
\int_{\mathbb{B}^2}  n_t\log n_t \, dV
\leq
M\log\left(
\frac{1}{4\pi e}
I(n_t)\right),
$$
for all $t\in (0,T^*)$.
This implies that
$$
-I(n_t)
\leq
-4e\pi \exp\left( 
\frac{1}{M} {\rm Ent}(n_t)\right).
$$
By using 1), we deduce that
$$
\frac{\partial}{\partial t} \left( {\rm Ent}(n_t) +\frac{\chi M^2}{4\pi} t\right)
\leq  \left[ -1+\frac{\chi M}{4\pi}\right] I(n_t)\leq
-4e\pi  \left[ 1-\frac{\chi M}{4\pi}\right]
\exp\left( 
\frac{1}{M} {\rm Ent}(n_t)\right).
$$
We set $\Psi(t)={\rm Ent}(n_t) +\frac{\chi M^2}{4\pi} t$. Then  $\Psi$ satisfies
$$
\Psi^{\prime}(t)\exp\left( 
-\frac{1}{M}\Psi(t) \right)
\leq
-4e\pi  \left[ 1-\frac{\chi M}{4\pi}\right]
\exp\left( 
-\frac{\chi M}{4\pi} t
\right).
$$
This differential inequality is easily integrated and leads to  \eqref{strongineqent}.
This completes the proof of the proposition.   
\end{proof}

Then, in the following proposition, we deduce some $L^q$-estimates  of  the solutions of the Keller-Segel system \eqref{kssyst}.

\begin{pro}\label{Lqdecay}
Let $(n_t)_{0<t<T^*}$ be a non-negative solution of the Keller-Segel system \eqref{kssyst} where $T^*\leq +\infty$  is the maximal existence time  of the solution. Then
\begin{enumerate}
\item
For all $t\in (0,T^*)$ and  all $1<q<\infty$, we have
\begin{equation}\label{weakdiffineqent2}
\frac{1}{q(q-1)}\frac{\partial}{\partial t} \vert\vert n_t\vert\vert_{L^q({\mathbb{B}^2})}^q \leq \left[ -1+\frac{\chi M}{8\pi}
\frac{(q+1)^2}{2q}\right] \int n_t^{q-2}\vert \nabla_{\mathbb{H}} n_t\vert^2\,dV.
\end{equation}
\item
If ${\chi} M\leq 4\pi h(q)$  with $h(q)=\frac{4 \, q}{(q+1)^2}$,
then
$$
\frac{\partial}{\partial t} \vert\vert n_t\vert\vert_{L^q({\mathbb{B}^2})}^q \leq 0.
$$
So, the map $t\rightarrow \vert\vert n_t\vert\vert_q^q $ is non-increasing on $(0,T^*)$ and we have
\begin{equation}\label{weakineqent2}
\vert\vert n_t\vert\vert_{L^q({\mathbb{B}^2})}\leq \vert\vert n_0\vert\vert_{L^q({\mathbb{B}^2})}
\end{equation}
 for all $t\in (0,T^*)$.
\end{enumerate}
\end{pro}
\begin{proof} 1)
By the classical multiplier method, we have that
$$
\frac{1}{q(q-1)}\frac{\partial}{\partial t} \vert\vert n_t\vert\vert_{L^q({\mathbb{B}^2})}^q 
=
- \int_{{\mathbb{B}^2} } n_t^{q-2}\vert \nabla_{\mathbb{H}} n_t\vert^2\, dV
+
\frac{\chi}{q}\int_{{\mathbb{B}^2} } n_t^{q+1}\,dV.
$$
By applying  the inequality \eqref{mug-tal}  to  $u=n_t^{(q+1)/2}$, we obtain that
$$
4\pi\left( \int_{{\mathbb{B}^2} } n_t^{q+1}\,dV\right)
\leq
 \frac{(q+1)^2}{4}M 
\left(
\int_{{\mathbb{B}^2} } n_t^{q-2}\vert \nabla_{\mathbb{H}} n_t\vert^2\,dV
\right),
$$
which implies \eqref{weakdiffineqent2}.
\\
2) This is deduced from  \eqref{weakdiffineqent2} of 1).
The proof is completed.
\end{proof}
\begin{rem}
By using the second statement above, if $ {\chi} M\leq {4\pi h(q)}<{4\pi}$ then in this case we have a $L^q$-norm control of the solution of the Keller-Segel system \eqref{kssyst}, which allow us to obtain a global well-posedness on $X_{T,q} \cap \mathcal{C} (\mathbb{R}^{+},L_+^1(\mathbb{B}^2))$ without weighted additional assumption on initial data in $L^1(\mathbb{B}^2)$.
\end{rem}

\end{document}